\documentclass[12pt]{amsart}
\usepackage{amssymb}
\usepackage{amsmath}
\usepackage{version}
\usepackage{color}
\usepackage{amsfonts}
\numberwithin{equation}{section}
\newtheorem{theorem}{Theorem}[section]
\newtheorem{lemma}[theorem]{Lemma}
\newtheorem{corollary}[theorem]{Corollary}

\newtheorem{definition}[theorem]{Definition}
\newtheorem{proposition}[theorem]{Proposition}

\newtheorem{remark}[theorem]{Remark}

\newcommand\beq{\begin{equation}}
\newcommand\ds{\displaystyle}
\newcommand\eeq{\end{equation}}

\newcommand\im{\mathrm {Im~}}
\newcommand\ii{\mathrm i}
\newcommand\al{\alpha}

\newcommand\la{\lambda}

\newcommand\D{\mathbb D}
\newcommand\T{\mathbb T}
\newcommand\C{\mathbb C}

\newcommand\R{\mathbb R}
\newcommand\N{\mathbb N}

\newcommand\Pick{\mathcal P}
\newcommand\Schur{\mathcal{S}}

\newcommand\df{\stackrel{\rm def}{=}}

\newcommand\nn{\nonumber}

\DeclareMathOperator{\schur}{schur}
\DeclareMathOperator\diag{diag}
\DeclareMathOperator\rank{rank}
\let\phi\varphi
\begin{document}
\title{The boundary Carath\'{e}odory-Fej\'{e}r interpolation problem}

\author{Jim Agler, Zinaida A. Lykova and N. J. Young}


\date{12 December 2010}

\begin{abstract}
We give a new solvability criterion for the {\em boundary Carath\'{e}odory-Fej\'{e}r problem}: given a  point $x \in \R$ and,  a finite set of target values, to construct a function $f$ in the Pick class such that the first few derivatives of $f$ take on the prescribed target values at $x$.  We also derive a linear fractional parametrization of the set of solutions of the interpolation problem.  The proofs are based on a reduction method due to Julia and Nevanlinna.
\end{abstract}

\thanks {Jim Agler was partially supported by National Science Foundation Grant DMS 0801259. N. J. Young was partially supported by London Mathematical Society Grant 4918 and by EPSRC Grant EP/G000018/1.}

\keywords{Pick class, boundary interpolation, parametrization, Hankel matrix, Schur complement.}

\maketitle
\markboth{Jim Agler, Zinaida A. Lykova and N. J. Young}
{The boundary Carath\'{e}odory-Fej\'{e}r interpolation problem}

\section{Introduction} \label{intro}

The {\em Carath\'{e}odory-Fej\'{e}r problem} \cite{CF} is to determine whether a given finite sequence of complex numbers comprises the initial Taylor coefficients of a function analytic in the unit disc $\D$ and having non-negative real part at each point of $\D$.
A ``boundary" version of the problem was introduced by R. Nevanlinna \cite{Nev1922} in 1922. Here one prescribes the first $n+1$ coefficients of the 
Taylor expansion of a function about a point of the {\em boundary} of $\D$.
In fact Nevanlinna studied the corresponding question for functions in 
the {\em Pick class} $\Pick$, which is defined to be the set of analytic functions $f$ on the upper half plane 
$$
\Pi\df \{z\in\C: \im z > 0\}
$$
such that $\im f \geq 0$ on $\Pi$. 
In this paper we shall study the following variant of the Carath\'{e}odory-Fej\'{e}r problem, in which the interpolation node lies on the real axis.\\

\noindent {\bf Problem  $\partial CF\Pick$}:
\quad {\em
Given a point $x \in \R$, a non-negative integer $n$  and numbers 
$a^{-1}, a^0, $ $\dots, $ $ a^{n} \in \C,$ determine whether there exists a function $f$ in the Pick class such that $f$ is analytic in a deleted neighbourhood of $x$ and
\beq \label{Prob_interp_LinPi}
L_k(f,x) = a^k, \qquad k=-1,0,1,\dots,n,
\eeq
where $L_k(f,x)$ is the $k$th Laurent coefficient of $f$ at $x$.
}

The nomenclature $\partial CF\Pick$ follows that introduced by D. Sarason in \cite{S}. 
Functions in the Pick class can have simple poles on the real axis, and so the boundary version of the  Carath\'{e}odory-Fej\'{e}r problem makes allowance for such poles.  

In this paper we give a new solvability criterion for Problem $\partial CF\Pick$ and an explicit parametrization of all solutions. We use only elementary methods; we believe that this will make our results easily accessible to engineers working in control and signal processing, where boundary interpolation questions arise (see \cite[Part VI]{bgr}, \cite{LimAnd}).

 We say that Problem $\partial CF\Pick$ is {\em solvable} if it has at least one solution, 
that is, if there exists a function $f\in\Pick$, meromorphic at $x$, such that the equations (\ref{Prob_interp_LinPi}) hold.  We say that the problem
is {\em determinate} if it has exactly one solution and is {\em indeterminate} if it has at least two solutions (and hence, by the convexity of the solution set, infinitely many solutions).  

Solvability of Problem $\partial CF\Pick$ is closely related to positivity of Hankel matrices.
For a finite indexed sequence $a=(a^{-1}, a^0, a^1, \dots,a^n)$ or $a=(a^0, a^1, \dots,a^n)$ and positive integer $m$ such that $2m-1 \leq n$ we define the associated Hankel matrix $H_m(a)$ by
\[
H_m(a) =[a_{i+j-1}]_{i,j=1}^{m}.
\]
We shall say that the Hankel matrix
$H_{m}(a)$ 
is {\em southeast-minimally positive} if $H_{m}(a) \ge 0$ and, for every~ $\varepsilon >0$, $H_{m}(a) - {\rm diag}\{0,0,\dots, \varepsilon\}$ is not positive.
We shall abbreviate ``southeast-minimally" to ``SE-minimally".

To state our main result we need the following notation. 
\begin{definition}\label{rho(a)} \rm
For a finite or infinite sequence  $a=(a^{-1},a^0, a^1, \dots )$ of complex numbers we define $\rho(a) \in \N \cup \{\infty \}$ by
\[
\rho(a) = \inf  \, \{k \ge 0: \im a^k \neq 0 \},
\]
 with the understanding that  $\rho(a)=\infty$ if all terms of the sequence are real.
\end{definition}

\begin{theorem} \label{main_theorem}
Let $a=(a^{-1},a^0,\dots,a^n)\in \C^{n+2}$ where $n\geq 1,\, a^{-1} \le 0$ and $\im a^0\geq 0$. \\
{\rm (1)}   If  $\rho(a)$  is finite and odd then Problem $\partial CF\Pick$  has no solution.   \\
{\rm (2)}  If $\rho(a) = 2m$ for some non-negative integer $m$ then Problem $\partial CF\Pick$  is solvable if and only if  $H_m(a) > 0$ and $\im a^{2m} > 0$.\\
{\rm (3)} If $\rho(a) = \infty$ then 
 for any odd positive integer $n$,
Problem $\partial CF\Pick$ is solvable if and only if  the associated Hankel matrix $H_{m}(a)$, $n=2m -1$, is either positive definite or SE-minimally positive. Moreover, the problem has a {\em unique} solution if and only if $H_m(a)$ is SE-minimally positive. \\
{\rm (4)} If $\rho(a) = \infty$ then for any even positive integer $n$,
Problem $\partial CF\Pick$ is solvable if and only if either the associated Hankel matrix $H_{m}(a)$, $n=2m$, is positive definite or both $H_{m}(a)$ is SE-minimally positive and $a^{n}$ satisfies 
\beq \label{form_a_2m_Main}
a^{n} = \left[ \begin{array}{cccc}
  a^m & a^{m+1} & \dots &a^{m+r-1}\end{array}\right] H_r(a)^{-1}
\left[\begin{array}{c} a^{m+1} \\ a^{m+2}\\ \cdot \\ a^{m+r} \end{array}\right] 
\eeq 
where $r=\rank H_m(a)$.
Moreover, the problem has a {\em unique} solution if and only if $H_m(a)$ is SE-minimally positive and $a^{n}$ satisfies equation {\rm (\ref{form_a_2m_Main})}.

\end{theorem}

Weaker notions of solvability are also of interest. It transpires, though, that  Problem $\partial CF\Pick$ has solutions 
in any reasonable weak sense  if and only if it
is solvable the strongest possible sense as defined here (the function is rational and analytic at the interpolation nodes); we plan to show this in a future paper. 

There is an extensive literature on boundary interpolation problems, well summarized in \cite[Notes For Part V]{bgr}.
We mention particularly papers of J. A. Ball and J. W. Helton \cite{bh},
D. Sarason \cite{S}, D. R. Georgijevi\'c \cite{Geo98,Geo05}, I.V. Kovalishina \cite{ko89} and V. Bolotnikov and A. Kheifets \cite{BolKh08}, and the books of J. A. Ball, I. C. Gohberg and L. Rodman \cite{bgr}, and of V. Bolotnikov and H. Dym \cite{BD}.
These authors make use of Krein spaces, moment theory, measure theory, reproducing kernel theory, realization theory and de Branges space theory.  They obtain far-reaching results, including generalizations to matrix-valued functions and to functions allowed to have a limited number of poles in a disc or half plane. See in particular a recent paper of  V. Bolotnikov \cite{Bol09} which treats the analogous problem for the Schur class and contains some results similar to ours; proofs are given in \cite{Bol10}.

 There is also a tradition of elementary treatments of interpolation problems (for example, \cite{Mar,power}).
In this paper we go back to the  Nevanlinna-Julia recursion
technique and derive a new solvability criterion and a parametrization of all solutions without the need for any Hilbert space notions beyond positivity of matrices. 

The paper is organised as follows. In Section \ref{reduction} we describe Julia's reduction procedure and its inverse, and give important properties of  reduction and augmentation in the Pick class.  In Section \ref{hankelreduce} we state and prove an identity which shows that Julia reduction of functions corresponds to Schur complementation of Hankel matrices. In Section \ref{interior_inter} we show that Problem $\partial CF\Pick$ has  a solution in the case that $a^0 \in\Pi$.
In Section \ref{posit_definite} we consider the case of  the boundary Carath\'{e}odory-Fej\'{e}r problem $\partial CF\Pick(\R)$ with all real $a^i$: we show that if the  Hankel matrix is positive definite then Problem $\partial CF\Pick(\R)$ is solvable for all  positive integers $n$. 
In Section \ref{relax_probl} we give a solvability criterion for a relaxation of Problem $\partial CF\Pick(\R)$, in which one equation in the interpolation conditions (\ref{Prob_interp_LinPi}) is replaced by an inequality.  
In Section \ref{atapoint} we prove Theorem \ref{main_theorem}, that is, we give  a solvability criterion for  Problem $\partial CF\Pick$.  In Section \ref{param} we give a parametrization of solutions of  Problem $\partial CF\Pick$, including an explicit formula for the solution in the determinate case (Theorem \ref{unique_sol_odd}). 
We also derive a linear fractional parametrization of the set of solutions of the interpolation problem. 
In Section \ref{nevanlinna} we discuss Nevanlinna's treatment of a  Carath\'{e}odory-Fej\'{e}r problem with the interpolation node at $\infty$.  We point out an inaccuracy in one of his statements about solvability of the problem and give a counter-example. In Section \ref{conclud} we close with a few comparisons of our results and methods with those of some other authors.

We shall write the imaginary unit as $\ii$, in Roman font, to have $i$ available for use as an index.   We denote the open unit disc by $\D$ and the  unit circle $\{z: |z|=1\}$  by $\T$.

\section{Julia reduction and augmentation in the Pick class} \label{reduction}

The main tool of this paper is a technique for passing from a function in the Pick class to a simpler one and back again.    Nevanlinna used a recursive technique for eliminating interpolation conditions at points on the real line \cite{Nev1}.  The reduction procedure (in the case of a function analytic at an interpolation point) is due to G. Julia \cite{Ju20}. It is analogous to the better-known ``Schur reduction" for functions in the Schur class \cite{Schur}.
See \cite{AY10} for a detailed updated account of reduction and augmentation in the Pick class.
 For any $x\in\R$ we shall say that a function $f\in\Pick$ {\em  is analytic at $x$} if $f$ extends to a function analytic in a neighbourhood of $x$.

\begin{definition} \label{defreduce} \rm
 (1)  For any non-constant function $f\in\Pick$ and any $x\in\R$ such that $f$ is analytic at $x$ and $f(x)\in\R$ we define the {\em reduction of $f$ at $x$} to be the function $g$ on $\Pi$ given by the equation
\beq \label{reducef}
g(z) = -\frac{1}{f(z)-f(x)} + \frac{1}{f'(x)(z-x)}.
\eeq

 (2)  For any $g\in\Pick$, any $x\in\R$ such that $g$ is meromorphic at $x$ and any $a^0\in\R$, $a^1 > 0$, we define the {\em augmentation of $g$ at $x$ by $a^0, a^1$} to be the function $f$ on $\Pi$ given by
\beq \label{augment2}
f(z)= a^0 + \frac{1}{\frac{1}{a^1(z-x)} -g(z)}.
\eeq
\end{definition}

\begin{remark} {\rm Let $g$ be a real rational function of degree $m$ and let $f$ be 
the augmentation of $g$ at $x$ by $a^0, a^1>0$. Then $f$ is a real rational function of degree $m+1$.}
\end{remark}

Here, as usual, the degree of a rational function $f = \frac{p}{q}$ is defined to be the maximum of the degrees of $p$ and $q$, where $p$, $q$ are polynomials in their lowest terms.

We shall need the following basic properties of the Pick class $\Pick$ (for example,
\cite[Proposition 3.1]{AY10}).
\begin{proposition} \label{posderiv}
Let $f$ be a non-constant function in the Pick class and let $x\in\R$.
\begin{enumerate}
\item[\rm(1)] If $f$ is analytic and real-valued at $x$ then $f'(x)>0$.
\item[\rm(2)] If $f$ is meromorphic and has a pole at $x$ then $f$ has a {\em simple} pole at $x$, with a negative residue.
\end{enumerate}
\end{proposition}

The important property of the operations of reduction and augmentation is that they preserve the Pick class. The following is contained in \cite[Theorem 3.4]{AY10}. 

\begin{theorem} \label{propfg}
Let $x \in\R$.
\begin{enumerate}
\item[\rm(1)] If a non-constant function $f\in\Pick$  is analytic and real-valued at $x$ then the reduction of $f$ at $x$ also belongs to $\Pick$ and is analytic at $x$.
\item[\rm(2)] If $g\in\Pick$ is analytic at $x$ and $a^0\in\R,\, a^1>0$ then the augmentation $f$ of $g$ at $x$ by $a^0,\, a^1$  belongs to $\Pick$, is analytic at $x$  and satisfies $f(x)=a^0, \, f'(x) = a^1.$  
\end{enumerate}
\end{theorem}

\noindent {\bf Notes.} (1) Under the assumptions of Definition \ref{defreduce} (1) we have $f'(x) > 0$.  For
suppose $f(\xi+\ii\eta)=u+\ii v$ with $\xi,\eta, u, v$ real: since $v>0$ on $\Pi$ and 
$v(x) =0$ we have  $ v_\eta (x)\geq 0$ and hence, by the Cauchy-Riemann equations, $u_\xi (x)\geq 0$. Furthermore the restriction of $v$ to a neighbourhood of $x$ in $\R$ attains its minimum  at $x$, and so $v_\xi=0$ at $x$.  Hence $f'(x)=(u_\xi+\ii v_\xi)(x) \geq 0$. \\

(2) A function $g$ obtained by reduction is not a general element of $\Pick$.
For if $g$ is the reduction of a Pick function meromorphic at $x$ then $g$ is analytic at $x$. \\

(3) Reduction and augmentation at a point of analyticity are of course inverse operations. \\ 

(4) Reduction and augmentation also apply to a wider class of functions in $\Pick$.

Specifically, if $f\in\Pick$ satisfies Carath\'eodory's condition at $x\in\R$, that is,
\[
\liminf_{z\to x}\frac{\im f(z)}{\im z} < \infty,
\]
then the reduction of $f$ at $x$ exists and belongs to $\Pick$.  The augmentation $f$ of $g\in\Pick$ at $x$ by $a^0\in\R, a^1>0$ lies in $\Pick$, satisfies Carath\'eodory's condition at $x$ and $f(x)=a^0, \ f'(x)\leq a^1$, where $f'(x)$ denotes the angular derivative of $f$ at $x$.  See \cite[Section 6]{AY10} for a full treatment.

It is simple to work out the relation between the Taylor series of a function and its reduction.
\begin{proposition} \label{taylorcoeffs}
Let $x\in \R$. Let $f$ be  analytic at $x$ and satisfy
 $f'(x) \neq 0$, and let $g$ be the reduction of $f$ at $x$ (so that $g$ is given by the equations  {\rm  (\ref{reducef})}). 
Let the Taylor expansions of $f, \, g$ about $x$ be
\[
f(z)= \sum_{j=0}^\infty f_j (z-x)^j, \quad g(z)= \sum_{j=0}^\infty g_j (z-x)^j.
\]
Then the Taylor coefficients $f_j$ and $g_j$ are related by
\beq\label{fandg}
\sum_{j=0}^\infty f_{j+1}\la^j \sum_{j=0}^\infty g_j\la^j= \frac{1}{f_1}\sum_{j=-1}^\infty f_{j+2}\la^j,
\eeq
or equivalently,
\beq \label{fandgbis}
\left[ \begin{array}{cccc} f_1 & 0 & 0 & \dots \\ f_2 & f_1 & 0 & \dots \\ f_3 & f_2 & f_1 & \dots \\ \cdot & \cdot & \cdot & \dots \end{array} \right]
 \left[ \begin{array}{c} g_0 \\ g_1 \\ g_2 \\ \cdot \end{array}\right] =
 \frac{1}{f_1} \left[ \begin{array}{c} f_2 \\ f_3 \\ f_4 \\ \cdot \end{array}\right].
\eeq
\end{proposition}

\begin{proof} Write $\la = z-x$.  Then
\begin{eqnarray*}
\sum g_j\la^j &=& - \frac{1}{f(z)-f_0} + \frac{1}{f_1\la}
   = \frac{f(z)-f_0-f_1\la}{f_1\la(f(z)-f_0)} \\
  &=& \frac{ \sum_{j=2} f_j\la^j}{ f_1 \sum_{j=1} f_j \la^{j+1}}.
\end{eqnarray*}
Hence 
\[
(f_1+f_2\la+f_3\la^2+\dots)(g_0+g_1\la+g_2\la^2+\dots)= \frac{1}{f_1}(f_2+f_3\la+f_4\la^2+\dots),
\]
which relation can also be expressed by the matricial formula (\ref{fandg}).   \end{proof}

\begin{lemma}\label{2augment} If $g$, $G \in \C[[z]]$ and $g(z) -G(z) = O(z^N)$ as $z \to 0$ then the augmentations $f, F$ of  $g$, $G$ respectively at $0$ by $a^0$ and $a^1 (\neq 0)$ satisfy 
\beq\label{fF}
f(z) -F(z) = (g(z) - G(z))(F(z)-a^0)(f(z)-a^0)
\eeq
and $f(z) -F(z) = O(z^{N+2})$ as $z \to 0$.
\end{lemma}
\begin{proof} A routine calculation gives equation (\ref{fF}), and so we have
\[
f(z) -F(z) = (g(z) - G(z))(F(z)-a^0)(f(z)-a^0) = O(z^N)O(z)O(z) =O(z^{N+2}).
\]
\end{proof}

\begin{lemma}\label{2augment_relax} Let  $x\in \R$ and $a^0,a^1,\dots,a^{n+2} \in \R$.  Let $f \in \Pick$  be analytic at  $x$ and
\beq 
\frac{f^{(k)}(x)}{k!} = a^k, \qquad k=0,1,\dots,n+1.
\eeq
Let $g$ be the reduction of $f$ at $x$, and let $G(z)=\sum_0^\infty g_j (z-x)^j$ be the reduction of $F(z)=\sum_0^{n+2} a^j (z-x)^j$ at $x$.
Then \beq \label{f_g_relax}
\frac{f^{(n+2)}(x)}{(n+2)!} - a^{n+2} =  (a^1)^2 \left(\frac{g^{(n)}(x)}{n!} - g_{n}\right).
\eeq
\end{lemma}
\begin{proof}
 By assumption,
\beq \label{f-F_2_relax}
f(z) -F(z) = f(z) - \sum_0^{n+2} a^j (z-x)^j =
\eeq 
\[
\left(\frac{f^{(n+2)}(x)}{(n+2)!} - a^{n+2} \right)(z-x)^{n+2} 
+O((z-x)^{n+3})
\]
as $z \to x$, since $f^{(k)}(x)/k! = a^k$ for $k=0,1,\dots,n+1$.
In view of Proposition \ref{taylorcoeffs},
\beq \label{g-G_relax}
g(z)-G(z) = g(z) - \sum_0^\infty g_j (z-x)^j = \left(\frac{g^{(n)}(x)}{n!} - g_{n}\right)(z-x)^{n} + O((z-x)^{n+1})
\eeq 
as $z \to x$. As in  Lemma \ref{2augment},
\beq \label{f-Fg-G_relax}
f(z) -F(z) = (g(z) - G(z))(F(z)-a^0)(f(z)-a^0) 
\eeq 
as $z \to x$.
It is easy to see that
$F(z)-a^0=a^1(z-x) + O((z-x)^2)$ and $f(z)-a^0=a^1(z-x) + O((z-x)^2)$ as $z \to x$. Therefore equations (\ref{f-Fg-G_relax}), (\ref{f-F_2_relax}) and (\ref{g-G_relax}) imply  
\beq 
\frac{f^{(n+2)}(x)}{(n+2)!} - a^{n+2} = \left(\frac{g^{(n)}(x)}{n!} - g_{n}\right) (a^1)^2. \qquad 
\eeq
 \end{proof}


Although Problem $\partial CF\Pick$ is formulated for functions {\em meromorphic} at $x$, in fact the crux of the problem is the analytic case.
\begin{proposition}\label{a-1toa0}
Problem $\partial CF\Pick$ is solvable 
if and only if $a^{-1} \le 0$ and there exists  a function
 $\tilde{f}$ in the Pick class such that $\tilde{f}$ is analytic at  $x$ and
\beq \label{interpcond_Taylor}
\frac{\tilde{f}^{(k)}(x)}{k!} = a^k, \qquad k=0,1,\dots,n.
\eeq
\end{proposition}
\begin{proof} 
Suppose that $f$ is a solution of Problem $\partial CF\Pick$.
By Proposition \ref{posderiv}, if $f$ is meromorphic and has a pole at $x$ then $f$ has a {\em simple} pole at $x$, with a negative residue. Therefore $a^{-1} \in \R$ and $a^{-1} \le 0$.  By the lemma of Julia (see for example \cite[Theorem 3.4]{AY10}) the function
\[
\tilde f(z) = f(z) - \frac{a^{-1}}{z-x}
\]
also belongs to the Pick class.  Clearly $\tilde f$ is analytic at $x$ and satisfies the equations (\ref{interpcond_Taylor}).

Conversely, suppose that $a^{-1} \le 0$ and $\tilde{f} \in\Pick$ satisfies the equations (\ref{interpcond_Taylor}).  The function $a^{-1}/(z-x)$ is in the Pick class, and so therefore is the function
 \[
f(z) = \frac{a^{-1}}{z-x} + \tilde{f}(z),
\]
which is thus a solution of Problem $\partial CF\Pick$.
 \end{proof}

\section{An identity for Hankel matrices} \label{hankelreduce}

The function-theoretic relation between a function $f$  and its reduction $g$ at a point corresponds to an identity relating their associated Hankel matrices.  In fact the matricial identity holds more generally -- for formal power series.

\begin{theorem} \label{hankident}
 Let
\[
f=\sum_{j=0}^\infty f_j z^j, \qquad g=\sum_{j=0}^\infty g_j z^j
\]
 be formal power series over a field $F$ with $f_1\neq 0$.  Suppose that $g$ is the reduction of $f$ at $0$, that is,
\beq \label{relfg}
 g= -\frac{1}{f-f_0} + \frac{1}{f_1 z}.
\eeq
Then
\begin{eqnarray} \label{hankelid}
\left[ \begin{array}{cccc} f_3 & f_4 & f_5 & \dots \\ f_4 & f_5 & f_6 & \dots \\
   \cdot & \cdot & \cdot & \dots \end{array} \right] &- &\frac{1}{f_1} \left[
   \begin{array}{c} f_2 \\ f_3 \\ \cdot \end{array}\right] \left[ \begin{array}{cccc}
  f_2 & f_3 & f_4 & \dots \end{array}\right] \nn  \\   & = & T_f \left[ \begin{array}{cccc} 
  g_1 & g_2 & g_3 & \dots \\ g_2 & g_3 & g_4 & \dots \\ \cdot & \cdot & \cdot & \dots\end{array} \right] T_f^T
\end{eqnarray}
where $T_f$ is the infinite Toeplitz matrix
\beq\label{deftoep}
T_f = \left[ \begin{array}{cccc} f_1 & 0 & 0 & \dots \\ f_2 & f_1 & 0 & \dots \\ \cdot & \cdot & \cdot & \dots \end{array}
\right]
\eeq
and the superscript $T$ denotes transposition.
\end{theorem}

\begin{remark}\label{convergence} {\rm There are no issues of convergence in equation (\ref{hankelid}).  Taking the $(i,j)$ entries of both sides, we find that the identity is equivalent to
\[
 f_{i+j+1} - \frac{ f_{i+1}f_{j+1}}{f_1} = \sum_{k=1}^i \sum_{\ell=1}^j f_k
  g_{i+j+1-k-\ell} f_\ell
\]
for all $i, j \geq 1$.}
\end{remark}

\begin{proof}
In the algebra $F[[z,w]]$ of formal power series over $F$ in commuting indeterminates $z,w$,
\begin{eqnarray*}
f(z)-f(w) &=& \sum_{j=1}^\infty f_j(z^j-w^j)  = (z-w)\sum_{j=1}^\infty f_j(z^{j-1}+z^{j-2}w+\dots+w^{j-1})\\
  &=& (z-w)\sum_{i,j=0}^\infty f_{i+j+1} z^iw^j,
\end{eqnarray*}
and hence
\[
 zw(f(z)-f(w)) = (z-w)\sum_{i,j=1}^\infty f_{i+j-1} z^iw^j.
\]
Since $g$ is the reduction of $f$ at $0$ we have
\begin{eqnarray*}
zw(g(z)-g(w)) &=& zw\left\{ -\frac{1}{f(z)-f_0}+\frac{1}{f_1z} + \frac{1}{f(w)-f_0} -\frac{1}{f_1w} \right\} \\
&=& \frac{zw}{(f(z)-f_0)(f(w)-f_0)}(f(z)-f(w)) - \frac{z-w}{f_1}.
\end{eqnarray*}
On multiplying through by $(f(z)-f_0)(f(w)-f_0)/(zw)$ and writing as a power series we find that
\[
  (z-w)\left\{\sum_{i,j=0}^\infty f_{i+j+1} z^iw^j - \frac{1}{f_1}\sum_{r=0}^\infty f_{r+1} z^r \sum_{s=0}^\infty f_{s+1} w^s \right\}=
\]
\[
(z-w)\sum_{r=0}^\infty f_{r+1}z^r \sum_{\al,\beta=1}^\infty g_{\al+\beta-1} z^\al w^\beta \sum_{s=0}^\infty f_{s+1}z^s =
\]
\[
(z-w) \sum_{i,j=1}^\infty z^iw^j  \sum_{r=0}^{i-1} \sum_{s=0}^{j-1} f_{r+1} g_{i-r+j-s-1}f_{s+1}.
\]
Cancel $z-w$ and equate the coefficients of $z^iw^j$ for $i,j \geq 1$ to obtain
\[
f_{i+j+1} - \frac{f_{i+1}f_{j+1}}{f_1} =\sum_{r=0}^{i-1} \sum_{s=0}^{j-1} f_{r+1} g_{i-r+j-s-1}f_{s+1},
\]
which is precisely equation (\ref{hankelid}) (put $r=k-1,\, s= \ell-1$ in Remark \ref{convergence} above).
 \end{proof}

Some terminology: for a matrix $M$, ``positive" means the same as ``positive semi-definite", and is written $M \geq 0$, whereas $M>0$ means that $M$ is positive definite. 


In a block matrix 
\[
M=\left[ \begin{array}{cc} A & B \\ C & D \end{array} \right],
\]
where $A$ is non-singular, the {\em Schur complement} of $A$ is defined to be $D-CA^{-1}B$.  By virtue of the identity
\beq\label{schid}
M=\left[ \begin{array}{cc} A & B \\ C & D \end{array} \right] = \left[\begin{array}{cc}1 & 0 \\ CA^{-1} & 1 \end{array}\right]
 \left[\begin{array}{cc}A & 0 \\ 0 & D-CA^{-1}B \end{array}\right]
 \left[\begin{array}{cc}1 & A^{-1}B \\ 0 & 1 \end{array}\right]
\eeq
it is clear that $\mathrm{rank}~(D-CA^{-1}B) = \mathrm{rank}~M - \mathrm{rank}~A$ and
\beq\label{det_schur}
{\det}\left[\begin{array}{cc} A & B\\ C & D \end{array}\right]= {\det}(A) \det (D - CA^{-1}B).
\eeq
 Furthermore, if $A>0$ then $M\geq 0$ if and only if $C=B^*$ and $D-CA^{-1}B\geq 0$.

For any $n \times n$ matrix $A = [a_{ij}]$ with $a_{11} \neq 0$ we define
$\schur A$ to be the Schur complement of $[a_{11}]$ in $A$. Thus 
$\schur A$ is $(n-1)$-square.

\begin{corollary} \label{schur_compl}
Let
\[
f=\sum_{j=0}^\infty f_j z^j, \qquad g=\sum_{j=0}^\infty g_j z^j
\]
 be formal power series over $\C$ with $f_1\neq 0$ and $f_0,f_1,\dots,f_n \in\R$, and let  $g$ be the reduction of $f$ at $0$.  Then, for any $n \geq 1$,
\beq \label{nbynhankel_schur}
\schur H_{n+1}(f) = 
 \left[ \begin{array}{cccc} f_1 & 0 & \dots & 0\\ f_2 & f_1 & \dots & 0 \\ 
\cdot & \cdot & \dots & \cdot \\
 f_n & f_{n-1} & \dots & f_1 \end{array}\right]  H_{n}(g)
\left[ \begin{array}{cccc} f_1 & f_2 & \dots & f_n\\ 0 & f_1 & \dots & f_{n-1}\\
 \cdot & \cdot & \dots & \cdot \\
 0 & 0 & \dots & f_1 \end{array}\right],
\eeq
\beq \label{schur_hank_det}
\det \left(\schur H_{n+1}(f) \right) = (f_1)^{2n} \det H_{n}(g),
\eeq
and
\beq \label{hank_det}
\det H_{n+1}(f) = (f_1)^{2n+1} \det H_{n}(g).
\eeq
\end{corollary}

\begin{proof} The left hand side of equation (\ref{hankelid}) is the Schur complement of the $(1,1)$ entry of the Hankel matrix
\[
\left[ \begin{array}{cccc} f_1 & f_2 & f_3 & \dots \\ f_2 & f_3 & f_4 & \dots \\ \cdot & \cdot & \cdot & \dots \end{array} \right].
\]
By virtue of Remark \ref{convergence}, if we take the $n\times n$ truncation of both sides in equation (\ref{hankelid}) we obtain the identity, for any $n \geq 1$,
\beq \label{nbynhankel}
\schur H_{n+1}(f) = \left[ \begin{array}{cccc} f_3 & f_4 &  \dots & f_{n+2}\\ f_4 & f_5 &  \dots & f_{n+3}\\
   \cdot & \cdot  & \dots &\cdot \\ f_{n+2} & f_{n+3} &  \dots & f_{2n+1}\end{array} \right] - \frac{1}{f_1} \left[
   \begin{array}{c} f_2 \\ f_3 \\ \cdot \\ f_{n+1} \end{array}\right] \left[ \begin{array}{cccc}
  f_2 & f_3 & \dots &f_{n+1}\end{array}\right]  =
\eeq
\[
 \left[ \begin{array}{cccc} f_1 & 0 & \dots & 0\\ f_2 & f_1 & \dots & 0 \\ 
\cdot & \cdot & \dots & \cdot \\
 f_n & f_{n-1} & \dots & f_1 \end{array}\right]
 \left[ \begin{array}{cccc}
  g_1 & g_2 &  \dots & g_n  \\ g_2 & g_3 &  \dots & g_{n+1}\\ 
 \cdot & \cdot &  \dots  &\cdot \\ g_n & g_{n+1} & \dots & g_{2n-1}   \end{array} \right] 
\left[ \begin{array}{cccc} f_1 & f_2 & \dots & f_n\\ 0 & f_1 & \dots & f_{n-1}\\
 \cdot & \cdot & \dots & \cdot \\
 0 & 0 & \dots & f_1 \end{array}\right]=
\]
\[
 \left[ \begin{array}{cccc} f_1 & 0 & \dots & 0\\ f_2 & f_1 & \dots & 0 \\ 
\cdot & \cdot & \dots & \cdot \\
 f_n & f_{n-1} & \dots & f_1 \end{array}\right]  H_{n}(g)
\left[ \begin{array}{cccc} f_1 & f_2 & \dots & f_n\\ 0 & f_1 & \dots & f_{n-1}\\
 \cdot & \cdot & \dots & \cdot \\
 0 & 0 & \dots & f_1 \end{array}\right]
\]
Moreover, if the last equation holds for each $n\geq 1$ then equation (\ref{hankelid}) is true.
One can see also that 
\[
\det \left(\schur H_{n+1}(f) \right) = (f_1)^{2n} \det H_{n}(g).
\]
By equation (\ref{det_schur}),
\[
\det H_{n+1}(f)= f_1 \det \left(\schur H_{n+1}(f)\right).
\]
In view of (\ref{schur_hank_det}),
\[
\det H_{n+1}(f)= f_1 \det \left(\schur H_{n+1}(f) \right) = f_1(f_1)^{2n} \det H_{n}(g).
\]
Therefore, 
\[
\det H_{n+1}(f) = (f_1)^{2n+1} \det H_{n}(g).\qquad 
\]
\end{proof}
\begin{corollary} \label{congruent}
Let
\[
f=\sum_{j=0}^\infty f_j z^j, \qquad g=\sum_{j=0}^\infty g_j z^j
\]
 be formal power series over $\C$ with $f_1\neq 0$ and $f_0,f_1,\dots,f_n \in\R$, and let  $g$ be the reduction of $f$ at $0$.  Then the $n\times n$ Hankel matrix 
$$H_n(g)=[g_{i+j-1}]_{i,j=1}^n$$ 
is congruent to the Schur complement of the $(1,1)$ entry in the $(n+1)\times(n+1)$ Hankel matrix 
$$H_{n+1}(f)=[f_{i+j-1}]_{i,j=1}^{n+1}.$$ 
 Consequently $H_{n+1}(f)> 0$ if and only if $f_1>0$ and $H_n(g)>0$.
\end{corollary}

\begin{proof} By Corollary \ref{schur_compl},
 the Schur complement of the $(1,1)$ entry of the $n+1$-square Hankel matrix $H_{n+1}(f)$ is 
$T_{f,n} H_n(g) T_{f,n}^T$ where $T_{f,n}$ is the $n\times n$ truncation of $T_f$. Since $T_{f,n}$ is a real matrix, $T_{f,n}^T$ is the adjoint of $T_{f,n}$.  Furthermore, since $f_1\neq 0$,  $T_{f,n}$ is invertible.  Thus the identity (\ref{nbynhankel_schur}) is a congruence between $H_n(g)$ and the Schur complement of the $(1,1)$ entry of $H_{n+1}(f)$.

The final statement follows from the facts that  (1) a hermitian matrix $\left[\begin{array}{cc} A & B\\ B^* & D \end{array}\right]$ is positive definite if and only if $A>0$ and the Schur complement of $A$  is positive definite, and (2) congruence preserves positive definiteness.
 \end{proof}
\section{Boundary-to-interior interpolation}\label{interior_inter}

In this section we prove solvability of Problem $\partial CF\Pick$ in the case that the target value $a^0$ has positive imaginary part, that is, lies in the interior of $\Pi$, not on the boundary.

Recall that we denote  by $\D$ the open unit disc $\{z:|z| < 1\}$
and by  $\overline{\D}$ the closed unit disc $\{z:|z| \le 1\}$. The {\em Schur class} $\Schur$ is the set of functions  analytic and have modulus bounded by 1 in $\D$.

\begin{theorem}\label{a0inPi}
Let $x\in \R$ and $a^{-1},a^0,a^1,\dots,a^n \in \C$,  
where $n\geq 1,\, a^{-1} \le 0$ and $\im a^0 > 0$. 
There exists  a function $f$ in the Pick class that is analytic in a deleted neighbourhood of $x$ and satisfies
\beq \label{interpcond_inter_LinPi}
L_k(f,x) = a^k, \qquad k=-1,0,1,\dots,n.
\eeq
\end{theorem}

In \cite[p.457]{bgr} the authors give an elegant way of treating boundary-to-interior interpolation problems. Their technique could doubtless be adapted to the present context, but in keeping with the elementary approach of this paper we prefer to give a ``bare hands" proof.

\begin{proof} It is enough to find a function $\tilde{f}$ in the Pick class that is analytic at  $x$ and satisfies
\beq \label{interpcond_TaylorinPi}
\frac{\tilde{f}^{(k)}(x)}{k!} = a^k, \qquad k=0,1,\dots,n.
\eeq
The function $f(z) = \frac{a^{-1}}{z-x} + \tilde{f}(z)$ will then satisfy conditions (\ref {interpcond_inter_LinPi}).

With the aid of the Cayley transform we may reduce the construction of such a function 
 $\tilde{f}$ in the Pick class to 
the construction of   $\varphi \in \Schur$ as in the following Proposition.
\end{proof}

\begin{proposition}\label{Schur_a0=0}
Let  $\tau \in \T $ and let $b^0 \in \D$, $b^1,\dots,b^n \in \C$. There exists  a function $\varphi$ in the Schur class  such that 
$\varphi$ is analytic at $\tau$ and
\beq \label{interpcond_k>=0_S}
\frac{\varphi^{(k)}(\tau)}{k!} = b^k, \qquad k=0,1,\dots,n.
\eeq
\end{proposition}

\begin{proof} By means of a M$\ddot{\rm o}$bius transformation we may reduce to the case that $b^0=0$. Let $q$ be the smallest index $k$, $ 1 \le k \le n$, such that $b^k \neq 0$ (if all $b^k = 0$ then of course $\varphi \equiv 0$ suffices). We will construct a function $\varphi(z)$ of the form
\beq \label{def_psi}
\varphi(z) = b^q (z-\tau)^q \psi(z),
\eeq
where  $\psi$ is analytic on $\D$. For such $\varphi $ and for all $j, q \le j \le n$,
\begin{eqnarray*}
\phi^{(j)}(z) &=&  b^q \sum_{m=0}^j {j\choose{m}}\frac{\partial^m}{\partial z^m}\left\{(z-\tau)^q \right\} \;\psi^{(j-m)}(z)\\
&=& b^q \frac{j!}{(j-q)!} \psi^{(j-q)}(z) + O(z-\tau)
\end{eqnarray*}

\noindent as $z \to \tau$. Thus, it will suffice to choose $\psi$ such that, for all $j, q \le j \le n$,
$$b^j= \frac{\phi^{(j)}(\tau)}{j!}= \frac{1}{j!} b^q \frac{j!}{(j-q)!}\psi^{(j-q)}(\tau).$$
Therefore we require $\psi$, analytic on $\D$, regular at the point $\tau$, such that
\beq \label{conditions_psi_1}
\psi^{(k)}(\tau) = \frac{b^{q+k} k!}{b^q} \qquad k=0,1,\dots,n-q,
\eeq
and, in view of equation (\ref{def_psi}) and the fact that we require $|\phi| \le 1$,
\beq \label{conditions_psi_2}
|\psi(z)| \le \frac{1}{|b^{q}||z-\tau|^q} \qquad \text{for all} \; z \in \T.
\eeq

\noindent Let
\beq \label{Q(z)}
Q(z) = d_0 + d_1 (z-\tau) + \frac{d_2}{2!} (z-\tau)^2 + \dots + \frac{d_{n-q}}{(n-q)!} (z-\tau)^{n-q}
\eeq
where $d_{k} = \frac{b^{q+k} k!}{b^q}$, $ k=0,1,\dots,n-q.$
For the next step we need a technical observation.
\begin{lemma}\label{h_N}
Let  $\tau \in \T $, let $m \in \N$ and, for $N \ge 1$, let
$$h_N(z)= 1-\left(1- \left(\frac{\tau +z}{2\tau}\right)^N \right)^m.$$
Then

{\rm (1)} $h_N(z)= 1 + O((z-\tau)^m )$ as $z \to \tau$ for any $N \ge 1$;

{\rm (2)} the functions $h_N(z),\; N=1,2,\dots,$ are uniformly bounded by $2^m +1$ on $\overline{\D}$;

{\rm (3)} $h_N(z) \to 0$ uniformly as $N \to \infty$ on  each compact subset of $\T \setminus \{ \tau \}$;

\end{lemma}

\begin{proof}
It is easy to see that $h_N(\tau)= 1$ and 
$h_N^{(p)}(\tau)= 0$ for all $1 \le p \le m-1$. Thus condition (1) is satisfied.

For all $z \in \overline{\D}$,
$$\left|\frac{\tau +z}{2\tau}\right|= \frac{|\tau +z|}{2} \le 1, \;\;\;\text{so} \;\;\; \left|\frac{\tau +z}{2\tau}\right|^N \le 1.$$
Hence, for all $N \ge 1$,
$$|h_N(z)|= \left|1-\left(1- \left(\frac{\tau +z}{2\tau}\right)^N \right)^m \right| \le 1 + \left|\left(1- \left(\frac{\tau +z}{2\tau}\right)^N \right)^m \right| \le 1+2^m .$$
For $z \in \T$ in any compact subset $K$ of $\T \setminus \{ \tau \}$, 
$$\left|\frac{\tau +z}{2\tau}\right|= \frac{|\tau +z|}{2} \le c < 1. $$
Therefore
$$h_N(z) = 1-\left(1- \left(\frac{\tau +z}{2\tau}\right)^N \right)^m\to 0$$ uniformly on $K$ as $N \to \infty$.
 \end{proof}

\noindent {\bf Conclusion of the proof of Proposition \ref{Schur_a0=0}.}~ Let
\beq \label{psi_N}
\psi_N (z) = Q(z)h_N(z),
\eeq
where $h_N, N=1,2, \dots,$ are defined in Lemma \ref{h_N} and are uniformly bounded on $\overline{\D}$ by $M= 2^m +1$. 
\noindent It is clear that condition (\ref{conditions_psi_1}), that is,
$$\psi^{(k)}(\tau) = d_k= \frac{b^{q+k} k!}{b^q} \qquad k=0,1,\dots,n-q,$$
 is satisfied for each function $\psi_N$.

Let us prove condition (\ref{conditions_psi_2}), that is, 
\beq 
|\psi(z)| \le \frac{1}{|b^{q}||z-\tau|^q} \qquad \text{for all} \; z \in \T.
\eeq
Choose $\varepsilon > 0$ such that 
$$|e^{i \varepsilon}\tau -\tau|^q < \frac{1}{|b^{q}| \| Q \|_{\infty} M}.$$
Then for all $\theta$ such that $|\theta| \le \varepsilon$ we have
$$ |\psi_N (e^{i\theta}\tau)| =| Q(e^{i\theta}\tau) h_N(e^{i\theta}\tau)| \le  \| Q \|_{\infty} M < \frac{1}{|b^{q}||e^{i\theta}\tau -\tau|^q}.$$
For all $\theta$ such that $|\theta| \ge \varepsilon$ and $ z=e^{i\theta}\tau$ we have
$$ |b^{q}(z -\tau)^q \psi_N (z)| = |b^{q}(z -\tau)^q Q(z)h_N(z)| \to 0$$
uniformly as $N \to \infty$.
Therefore,  condition (\ref{conditions_psi_2}) 
holds for $\psi = \psi_{N}$ for  sufficiently large $N$.

Hence, if we choose $\varphi(z) = b^q (z-\tau)^q \psi_{N}(z)$
for  sufficiently large $N$, we obtain $\varphi \in \Schur$ that satisfies (\ref{Schur_a0=0}).
 \end{proof}

\section{The boundary Carath\'{e}odory-Fej\'{e}r problem with positive definite Hankel matrix} \label{posit_definite}

In this section  we consider the special case of  the boundary Carath\'{e}odory-Fej\'{e}r problem in which $a^{-1}=0$ and all  $a^i$ are real.\\

\noindent {\bf Problem $\partial CF\Pick(\R)$} \quad {\em
Given a point $x\in \R$ and $a^0,a^1,\dots,a^n \in \R$, find a function $f$ in the Pick class such that $f$ is analytic at  $x$ and
\beq \label{interpcond_anal}
\frac{f^{(k)}(x)}{k!} = a^k, \qquad k=0,1,\dots,n.
\eeq
}

\begin{theorem} \label{H_m(a)>0_HD} Let $n$ be a  positive integer, let $a=(a^0,a^1,\dots,a^n)$ and let $m$ be the integer part of $\tfrac 12 (n+1)$.
Suppose the Hankel matrix  $H_m(a)>0$.  Then Problem $\partial CF\Pick(\R)$ is solvable and indeterminate. Moreover, the solution set of the problem contains a real rational function of degree equal to the rank of $H_m(a)$.
\end{theorem}

\begin{proof}  In the case $n=1$ we have $m =1$ and $H_m(a)=[a^1]>0$.
The real rational function $f$ of degree 1 
\[
f(z) = a^0 +  \frac{a^1(z-x)}{1-c(z-x)}
\]
 is a solution of Problem $\partial CF\Pick(\R)$ for every $c \in\R$, and so the problem  is solvable and indeterminate, and its solution set contains a real rational function of degree equal to $\mathrm{rank}~H_1(a)=1$. 

Consider the case $n=2$.  Again $m=1$ and $H_1(a)=[a^1]>0$. We want to prove that  there exists $f\in\Pick$ such that $f(x)=a^0,\, f'(x)=a^1,\,  \tfrac 12 f''(x)=a^2$. If $f$ is a function with the desired properties then the reduction of $f$ at $x$ takes the value $a^2/(a^1)^2$ at $x$.  Accordingly let $g$ be any function in $\Pick$ that is analytic at $x$ and satisfies $g(x)=a^2/(a^1)^2$, e.g.
\[
g(z) = \frac{a^2}{(a^1)^2}+ c(z-x), \;\;c \geq 0. 
\]
Let $f$ be the augmentation of $g$ at $x$ by $a^0,\, a^1$.  Then $f$ is an interpolating function as required.  In particular, if $c$ is chosen to be zero then we obtain the interpolating function 
\[
f(z)= a^0 + \frac{(a^1)^2(z-x)}{a^1-a^2(z-x)}.
\]
The function  $f$ is a real rational function of degree 1, and $\mathrm{rank}~H_1(a)=1$.
Thus the assertions of the theorem hold when $n=1$ or $2$, that is, when $m=1$.

 Suppose the statement of the theorem holds for some  $m \geq 1$; we shall prove it holds  for $m+1$, in which case $n=2m+1$ or $2m+2$. 
Let  $H_{m+1}(a) > 0$; then $a^1 > 0$. 
Let $\sum_0^\infty g_j z^j$ be the reduction of $\sum_0^n a^j z^j$ at $0$.  
By Corollary \ref{congruent}, $H_m(g_0,g_1,\dots,g_{2m-1})$ is congruent to  $\schur H_{m+1}(a)$ and therefore is positive definite.  Now $n-2$ is either $2m-1$ or $2m$, and so,
 by the inductive hypothesis, there exist at least two functions $g \in \Pick$ such that 
\[
\frac{g^{(k)}(x)}{k!} = g_k, \qquad k=0,1,\dots, n-2,
\]
one of these functions being a real rational function of degree equal to the rank of $H_{m}(g_0,g_1,$ $\dots,g_{2m-1})$. For both $g$,
\[ 
g(z) - \sum_0^\infty g_j (z-x)^j = O((z-x)^{n-1}).
\]
Let $f$ be the augmentation of $g$ at $x$ by $a^0,a^1$; then $f\in \Pick$ and $f$ is analytic at $x$.
Note that $\sum_0^{n} a^j (z-x)^j$ is the augmentation of $\sum_0^\infty g_j (z-x)^j $ at $x$ by $a^0,a^1$.  Hence, by Lemma \ref{2augment},
\[
 f(z) - \sum_0^{n} a^j (z-x)^j = O((z-x)^{n+1}),
\]
and so
\[
f^{(k)}(x)/k! = a^k \mbox{  for } k=0,1,\dots,n.
\]
We obtain two solutions of 
Problem $\partial CF\Pick(\R)$, one of them a real rational function of degree equal to 
$$
{\rm degree} (g) +1 = \mathrm{rank}~H_m(g_0,g_1,\dots,g_{2m-1}) +1=\mathrm{rank}~H_{m+1}(a).
$$
By induction the statement of the theorem holds for all $m\geq 1$.
 \end{proof}

\section{A relaxation of the boundary Carath\'{e}odory-Fej\'{e}r problem} \label{relax_probl}

It is not far from the truth that Problem $\partial CF\Pick(\R)$ (for $n=2m-1$) is solvable if and only if its associated Hankel matrix $H_m(a)$ is positive. However, consider the case that $a^1=0$. Then 
$H_m(a) \ge 0$ if and only if 
$a^2 = a^3 = \dots = a^{n-1} = 0$ and $a^{n} \ge 0$, whereas 
Problem $\partial CF\Pick(\R)$ is solvable if and only if 
$a^2 = a^3 = \dots = a^{n} = 0$. It turns out that positivity of $H_m(a)$ characterises solvability of a relaxed version of the problem:\\

\noindent {\bf Problem $\partial CF\Pick'(\R)$} \quad {\em
Given a point $x\in \R$ and $a^0,a^1,\dots,a^n \in \R$, find a function $f$ in the Pick class such that $f$ is analytic at  $x$,
\beq \label{interpcondRel}
\frac{f^{(k)}(x)}{k!} = a^k, \qquad k=0,1,\dots,n-1, \;\;{\text and } \;\;  \frac{f^{(n)}(x)}{n!} \le a^n.
\eeq
}

\begin{theorem} \label{probprimeHD} 
Let $n$ be an odd positive integer, $a=(a^0,\dots,a^n)\in \R^{n+1},\ x\in\R$.

{\rm (1)} Problem $\partial CF\Pick'(\R)$ is solvable if and only if the associated Hankel matrix $H_m(a)$ is positive, where $m = \tfrac 12 (n+1)$.

{\rm (2)} The solution set of Problem $\partial CF\Pick'(\R)$, if nonempty, contains a real rational function of degree no greater than the rank of the associated Hankel matrix.

{\rm (3)} The problem is determinate if and only if the associated Hankel matrix is positive and singular.
\end{theorem}
\begin{proof}
First consider the case that $a^1=0$.  By Proposition \ref{posderiv}, any $f\in\Pick$ for which $f'(x)=0$ is a constant, and so Problem $\partial CF\Pick'(\R)$ is solvable if and only if $a^2=\dots=a^{n-1} =0$ and $a_n \geq 0$.  As we observed above, these are also the conditions that $H_m(a) \geq 0$, and so (1) is true when $a^1=0$.  It is easily seen that (2) and (3) also hold in this case.  We therefore suppose henceforth that $a^1 >0$.

When $m =1$ we have $H_m(a)=[a^1]>0$ and Poblem $\partial CF\Pick'(\R)$ has infinitely many solutions, to wit
\beq\label{many}
f(z) = a^0 + c(z-x) \quad \mbox{ where }0 < c \leq a^1.
\eeq
Accordingly statement (1) of the theorem holds when $m=1$.  It is clear from equation (\ref{many}) that (2) and (3) also hold in this case.

 Suppose sufficiency holds in statement (1) of the theorem for some $m \geq 1$; we shall prove it holds for $m +1$.

 Suppose $H_{m+1}(a)\ge 0$.
If  $H_{m+1}(a) > 0$
then by Theorem \ref{H_m(a)>0_HD},   Problem $\partial CF\Pick(\R)$ is solvable.  {\em A fortiori}, so is  Problem $\partial CF\Pick'(\R)$.  We may therefore suppose that $H_{m+1}(a)$ is positive and singular.

Let $G(z)=\sum_0^\infty g_j (z-x)^j$ be the reduction of $F(z)=\sum_0^{2m+1} a^j (z-x)^j$ at $x$. By Corollary \ref{congruent},   $H_m(g_0,g_1,\dots,g_{2m-1})$ is congruent to  $\schur H_{m+1}(a)$. Since $H_{m+1}(a)$ is positive and singular, so is  $H_m(g_0,g_1,\dots,g_{2m-1})$.
By the inductive hypothesis there exists a function $g \in \Pick$ 
such that 
\beq\label{propg}
\frac{g^{(k)}(x)}{k!} = g_k, \qquad k=0,1,\dots, n-1, \; \; \text{and} \;\; \;
\frac{g^{(n)}(x)}{n!} \le g_{2m-1}.
\eeq
Hence
\begin{align} \label{g-G}
g(z)-G(z) &= g(z) - \sum_0^\infty g_j (z-x)^j \nn \\ &= \left(\frac{g^{(2m-1)}(x)}{(2m-1)!} - g_{2m-1}\right)(z-x)^{2m-1} + O((z-x)^{2m}).
\end{align}
 $F$ is the augmentation of $G$ at $x$ by $a^0,a^1$;
let $f$ be the augmentation of $g$ at $x$ by $a^0,a^1$.
Thus $f\in \Pick$, $f$ is analytic at $x$ and, by Lemma \ref{2augment},
\beq \label{f-F_1}
 f(z)-F(z) = f(z) - \sum_0^{2m+1} a^j (z-x)^j = O((z-x)^{2m+1}).
\eeq 
Therefore
$f^{(k)}(x)/k! = a^k$ for $k=0,1,\dots,2m$.
By Lemma \ref{2augment_relax}, we have
\beq \label{f_g_relax_2}
\frac{f^{(2m+1)}(x)}{(2m+1)!} - a^{2m+1} = \left(\frac{g^{(2m-1)}(x)}{(2m-1)!} - g_{2m-1}\right) (a^1)^2.
\eeq
Since $\frac{g^{(2m-1)}(x)}{(2m-1)!} \le g_{2m-1}$, we have
$f^{(2m+1)}(x)/(2m+1)! \le a^{2m+1}$. 
Thus $f$ is a solution of $\partial CF\Pick'(\R)$.  By induction, sufficiency holds in statement (1).

Now suppose that necessity holds in (1) for some $m\geq 1$. Suppose a problem $\partial CF\Pick'(\R)$ has a solution $f\in \Pick$ for some $x\in \R$ and $a=(a^0, \dots, a^{2m+1})$.
Let $g$ be the reduction of $f$ at $x$. Then $g\in \Pick$, and so, by the inductive 
 hypothesis, $H_m(g_0,g_1,\dots,g_{2m-1}) \ge 0$ where $g_k = g^{(k)}(x)/k!$~ for $k=0,1, \dots, 2m-1$. By Corollary \ref{congruent}, 
 $H_m(g_0,g_1,\dots,g_{2m-1})$ is congruent to $\schur H_{m+1}(f)$ where
\[
 H_{m+1}(f) = \left[ \frac{f^{(i+j-1)}(x)}{(i+j-1)!}\right]_{i,j=1}^{m+1}.
\]
Thus $H_{m+1}(f) \ge 0$. Since $f$ is a solution of Problem $\partial CF\Pick'(\R)$,
\[
H_{m+1}(a) = H_{m+1}(f) +  \diag \{0,\dots, 0, a^{2m+1} - f^{(2m+1)}(x)/(2m+1)! \}  \ge 0. 
\]
Thus, by induction, necessity also holds in statement (1) of the theorem, and so (1) is proved.

We know that (2) holds when $m=1$; suppose it holds for some $m\geq 1$ and consider a solvable problem  $\partial CF\Pick'(\R)$ with $a=(a^0, \dots,a^{2m+1}$.  By (1), $H_{m+1}(a) \ge 0$.  Once again let $G(z)=\sum_0^\infty g_j (z-x)^j$ be the reduction of $F(z)=\sum_0^{2m+1} a^j (z-x)^j$ at $x$. Again by Corollary \ref{congruent}, $H_m(g_0,g_1,\dots,g_{2m-1})$  is positive.  By the inductive hypothesis
there is a real  rational function $g$ of degree no greater than  $\rank H_m(a)$ that satisifes the relations (\ref{propg}).   Let $f$ be the augmentation of $g$ at $x$ by $a^0,a^1$; then $f$ is a real rational function of degree equal to 
$$ 
1 +\;\;{\rm degree}(g) \le 1 +\;\;\rank H_m(g_0,g_1,\dots,g_{2m-1}) = \rank H_{m+1}(a).
$$
Exactly as in the proof of necessity in (1), $f$ is a solution of Problem  $\partial CF\Pick'(\R)$.
Thus, by induction, (2) holds for all $m$.

Necessity holds in statement (3) by Theorem \ref{H_m(a)>0_HD}: if $H_m(a) >0$ then  Problem $\partial CF\Pick(\R)$ is indeterminate, and so {\em a fortiori} $\partial CF\Pick'(\R)$ is indeterminate.

We know that sufficiency holds in statement (3) when $m=1$.  Suppose it  holds for some $m \geq 1$; we shall prove it holds for $m +1$. 

Let $H_{m+1}(a)$ be positive and singular for some $a=(a^0,\dots,a^{2m+1})$.
Assume that functions
$f_1$ and $f_2$ in $\Pick$ are solutions of the problem $\partial CF\Pick'(\R)$.
Let $G(z)=\sum_0^\infty g_j (z-x)^j$ be the reduction of $F(z)=\sum_0^{2m+1} a^j (z-x)^j$ at $x$. By Corollary \ref{congruent},   $H_m(g_0,g_1,\dots,g_{2m-1})$ is congruent to  $\schur H_{m+1}(a)$. Since $H_{m+1}(a)$ is positive and singular, so is  $H_m(g_0,g_1,\dots,g_{2m-1})$.
Therefore, by Lemma \ref{2augment_relax},  the reductions $g^1$ and $g^2$ at $x$ of $f_1,\ f_2$ respectively are  solutions of the problem $\partial CF\Pick'(\R)$ for the data $x\in \R$ and  $g_0,g_1,\dots,g_{2m-1} \in \R$.
By the inductive hypothesis the solution of this problem is unique, and so $g^1 = g^2$. Since $f_1, \ f_2$ are both equal to the augmentation of this function at $x$ by  $a^0$, $a^1$ we have $f_1=f_2$.
Thus, by induction, for any $m\geq 1$, if $H_m(a)$ is positive and singular then the problem $\partial CF\Pick'(\R)$ is determinate. 
 \end{proof}

The idea of proving a solvability result for a boundary interpolation problem by first solving a relaxed problem has been used by several authors in the context of the boundary Nevanlinna-Pick problem; for example, D. Sarason \cite{S}.

\section{A criterion for the boundary Carath\'{e}odory-Fej\'{e}r problem} \label{atapoint}

In this section  we prove Theorem \ref{main_theorem}, that is, we give  a solvability criterion for  Problem $\partial CF\Pick$:  {\em
Given a point $x\in \R$ and $a^{-1},a^0,a^1,\dots,a^n \in \C$, where $n$ is a non-negative integer, to find a function $f$ in the Pick class such that $f$ is analytic in a deleted neighbourhood of $x$ and
\beq \label{interpcond}
L_k(f,x) = a^k, \qquad k=-1,0,1,\dots,n.
\eeq
}

To start with we consider the case where all terms of the sequence are real and  deduce a solvability theorem for Problem $\partial CF\Pick(\R)$ from Theorem \ref{probprimeHD}.
Things are somewhat diffrent for even and odd $n$; we start with odd $n$.

\begin{theorem} \label{sarasonHD_odd} Let $m$ be a positive integer and $a=(a^0,\dots,a^{2m-1})\in\R, \ x\in\R$. Then

{\rm (1)} Problem $\partial CF\Pick(\R)$ 
is solvable if and only if the associated Hankel matrix $H_{m}(a)$ is either positive definite or SE-minimally positive.  

{\rm (2)} The solution set of the problem, if non-empty, contains a real rational function of degree equal to the rank of $H_m(a)$.

{\rm (3)} The problem is determinate if and only if $H_m(a)$ is SE-minimally positive.
\end{theorem}
 
\begin{remark} \rm  A SE-minimally positive Hankel matrix $H$ is  automatically singular. For if $H >0$ then all the leading principal minors of $H$ are positive, and so they remain positive for sufficiently small perturbations of the SE-corner entry of $H$. 
\end{remark}
\begin{proof} 
{\em Necessity} of (1).  Suppose that Problem $\partial CF\Pick(\R)$ has a solution $f\in\Pick$ but that its  Hankel matrix $H_m(a)$ is neither positive definite nor SE-minimally positive.  {\em A fortiori} $f$ solves Problem $\partial CF\Pick'(\R)$, and so, by Theorem \ref{probprimeHD}, $H_m(a)\geq 0$.  Since $H_m(a)$ is not positive definite, $H_m(a)$ is singular, and so Problem $\partial CF\Pick'(\R)$ has the {\em unique} solution $f$.  Since $H_m(a)$ is not SE-minimally positive there is some positive $a^{n}{'} < a^{n}$ such that $H_m(f)\geq 0$, where $H_m(f)$ is the matrix obtained when the $(m,m)$ entry $a^{n}$, $n=2m-1$, of $H_m(a)$ is replaced by $a^{n}{'}$.  Again by Theorem \ref{probprimeHD}, there exists $h\in\Pick$  such that $h$ is analytic at $x$, $h^{(k)}(x)/k! = a^k$ for $k=0,1,\dots,n-1$ and $h^{(n)}(x)/n! \le a^{n}{'}< a^{n}$.
 In view of the last relation  we have $h\neq f$, while clearly $h$ is a solution of Problem $\partial CF\Pick'(\R)$, as is $f$.  This contradicts the uniqueness of the solution $f$.  Hence if the problem is solvable then either $H_m(a)>0$ or $H_m(a)$ is SE-minimally positive.

{\em Sufficiency} of (1). In the case $H_m(a)>0$,
by Proposition \ref{H_m(a)>0_HD},  the problem $\partial CF\Pick(\R)$ is solvable. Suppose that $H_m(a)$ is SE-minimally positive.  By Theorem \ref{probprimeHD} there is an $f\in\Pick$ such that $f^{(k)}(x)/k! = a^k$ for $k=0,1,\dots,n-1$ and $f^{(n)}(x)/n! \le a^{n}$.
If in fact $f^{(n)}(x)/n! < a^{n}$, then consider the matrix $H_m(f)$ obtained when the $(m,m)$ entry $a^n$ of $H_m(a)$ is replaced by $f^{(n)}(x)/n!$.  Since $H_m(f)$ is the Hankel matrix of a problem $\partial CF\Pick'(\R)$ that is solvable (by $f$), we have $H_m(f)\geq 0$ by Theorem \ref{probprimeHD}, and so $H_m(a)$ majorises the non-zero positive diagonal matrix $\mathrm{diag}~\{0,\dots, a^{n}- f^{(n)}(x)/n!\}$, contrary to hypothesis.  Thus $f^{(n)}(x)/n! = a^{n}$, that is, $f$ is a solution of Problem $\partial CF\Pick(\R)$.  We have proved (1). Moreover, since $H_m(a)$ is singular, $f$ is the unique solution of Problem $\partial CF\Pick'(\R)$, hence is real rational of degree at most $\rank H_m(a)$.

 (2) If  $\partial CF\Pick(\R)$ has a solution then 
the associated Hankel matrix $H_{m}(a)$, $n=2m -1$, is positive definite or SE-minimally positive.  If $H_m(a) > 0$, by  Theorem \ref{H_m(a)>0_HD}, there is a real  rational solution of degree  $\rank H_m(a)$. If $H_m(a)$ is SE-minimally positive, by Theorem \ref{probprimeHD}, $\partial CF\Pick'(\R)$ is determinate and the solution $f$ of $\partial CF\Pick'(\R)$ is a real  rational function of degree  at most $\rank H_m(a)$. As is shown above, $f$ is also the solution of $\partial CF\Pick(\R)$. 
By Kronecker's theorem \cite[Theorem 3.1]{peller}, the real rational function $f$ has degree equal to $\rank H_m(a)$.
Thus (2) holds for all $m$.

 Theorem \ref{H_m(a)>0_HD} proves the necessity of statement (3): $H_m(a) >0$ implies that  Problem $\partial CF\Pick(\R)$ is indeterminate. As in statement (2),  by Theorem \ref{probprimeHD}, if $H_m(a)$ is SE-minimally positive  then Problem $\partial CF\Pick(\R)$ is determinate. 
 \end{proof}

We now turn to the solvability question for even $n$.
\begin{lemma} \label{H_SEmin_posit_rank} Let $H_m(a)$ be the Hankel matrix  associated with Problem $\partial CF\Pick(\R)$.
If $H_m(a)$ is SE-minimally positive and has rank $r < m$
then $H_r(a)$ is nonsingular.
\end{lemma}

\begin{proof} By Theorem \ref{sarasonHD_odd}, there exists a unique solution $f \in \Pick$ of Problem $\partial CF\Pick(\R)$:
\[
\frac{f^{(k)}(x)}{k!} = a^k, \qquad k=0,1,\dots,2m-1,
\]
and $f$ is real rational function of degree $r$.

Suppose $H_r(a)$ is singular of rank $r'<r$. Since $H_m(a)$ is SE-minimally positive, $H_r(a)$ is positive. By Theorem \ref{probprimeHD}, there is a unique solution $g \in \Pick$ of the relaxed problem 
\beq \label{relax_2r-1}
\frac{g^{(k)}(x)}{k!} = a^k, \qquad k=0,1,\dots,2r-2, \;\;\text{and } \;\;  \frac{g^{(2r-1)}(x)}{(2r-1)!} \le a^{2r-1}.
\eeq
Moreover, the solution $g$ is a real rational function of degree no greater than $r'$. We have a contradiction since $f$ is also a solution of Problem \ref{relax_2r-1} and $f$ is a real rational function of degree $r$, which is strictly greater than $r'$. Hence $H_r(a)$ is nonsingular.
 \end{proof}

\begin{proposition} \label{H_SEmin_posit_a_2m} Let $a=(a^0, \dots, a^{2m})\in\R^{2m+1}$.
Suppose that $H_m(a)$ is SE-minimally positive and $a^1 >0$. Then
Problem $\partial CF\Pick(\R)$ 
is solvable if and only if
\beq \label{form_a_2m}
a^{2m} = \left[ \begin{array}{cccc}
  a^m & a^{m+1} & \dots &a^{m+r-1}\end{array}\right] H_r(a)^{-1}
\left[\begin{array}{c} a^{m+1} \\ a^{m+2}\\ \cdot \\ a^{m+r} \end{array}\right] 
\eeq 
where $r=\rank H_m(a)$.
Moreover, the solution $f$ is unique and is a real rational function of degree equal to $r$.
\end{proposition}
\begin{proof} 
By assumption, $H_m(a)$ is SE-minimally positive and $a^1 >0$. Thus,
 by Theorem \ref{sarasonHD_odd}, there exist a unique $F\in\Pick$  such that $F^{(k)}(x)/k! = a^k$ for $k=0,1,\dots,2m-1$ and $F$ is a real rational function of degree equal to $r=\rank H_m(a)$, $1 \le r \le m$.
 
By a result of Kronecker \cite[Theorem 3.1]{peller}, the infinite Hankel matrix
\[
H(F) = \left[ \begin{array}{cccccccc} 
F_1 & F_2 & \dots  & F_{r} & \dots  & F_{m} & F_{m+1} & \dots\\ 
\cdot & \cdot & \dots & \cdot & \dots & \cdot & \cdot & \dots \\
 F_r & F_{r+1} & \dots & F_{2r-1} & \dots  &   F_{m+r-1} & F_{m+r} & \dots\\ 
\cdot & \cdot & \dots & \cdot & \dots  & \cdot & \cdot & \dots \\
 F_{m} & F_{m+1} &\dots & F_{m+r-1} & \dots  & F_{2m-1} & F_{2m} & \dots\\ 
\cdot & \cdot & \dots & \cdot & \dots  & \cdot & \cdot & \dots \\
\end{array} \right].
\]
has rank $r$. Therefore any square $(r+1)\times (r+1)$ submatrix is singular. By Lemma \ref{H_SEmin_posit_rank}, $H_r(F) >0$. Hence, by equation (\ref{det_schur}), 
\begin{align*}
0 &= {\det}\left[ \begin{array}{cccccccc} 
F_1 & F_2 & \dots  & F_{r}  & F_{m+1} \\ 
\cdot & \cdot & \dots & \cdot & \cdot \\
 F_r & F_{r+1} & \dots & F_{2r-1} &  F_{m+r}\\ 
 F_{m} & F_{m+1} &\dots & F_{m+r-1}  & F_{2m}\\
\end{array} \right] 
\\
  &= \det H_r(F)\left( F_{2m} -
\left[ \begin{array}{cccc}  F_{m} & F_{m+1} &\dots & F_{m+r-1}\end{array} \right] H_r(F)^{-1} \left[\begin{array}{c} F_{m+1} \\ F_{m+2}\\ \cdot \\ F_{m+r} \end{array}\right] \right).
\end{align*}
Since
$F^{(k)}(x)/k! = a^k$ for $k=0,1,\dots,2m-1$, we have
\[
F_{2m} = \left[ \begin{array}{cccc}
  a^m & a^{m+1} & \dots &a^{m+r-1}\end{array}\right] H_r(a)^{-1}
\left[\begin{array}{c} a^{m+1} \\ a^{m+2}\\ \cdot \\ a^{m+r} \end{array}\right] .
\]
Thus Problem $\partial CF\Pick(\R)$ 
is solvable if and only if equation (\ref{form_a_2m}) holds.
\end{proof}

\begin{theorem} \label{sarasonHD_even} Let $n$ be an even positive integer. Then

{\rm (1)} Problem $\partial CF\Pick(\R)$ 
is solvable if and only if either the associated Hankel matrix $H_{m}(a)$, $n=2m$, is positive definite or  $H_{m}(a)$ is SE-minimally positive of rank $r \ge 1$ and $a^{n}$ satisfies {\rm (\ref{form_a_2m})}.  

{\rm (2)} The solution set of the problem, if non-empty, contains a real rational function of degree equal the rank of the Hankel matrix.

{\rm (3)} The problem is determinate if and only if the Hankel matrix is SE-minimally positive and $a^{n}$ satisfies {\rm (\ref{form_a_2m})}.
\end{theorem}

\begin{proof} 
{\em Necessity} of (1).  Suppose that Problem $\partial CF\Pick(\R)$ has a solution $f\in\Pick$  such that $f^{(k)}(x)/k! = a^k$ for $k=0,1,\dots,2m$. This 
$f\in\Pick$ is also a solution of Problem $\partial CF\Pick(\R)$ for $n= 2m-1$.
The  Hankel matrix $H_m(a)$ for Problem $\partial CF\Pick(\R)$ with $n=2m$ and with $n= 2m-1$ is the same.
By Theorem \ref{sarasonHD_odd}, $H_m(a)$ is  positive definite or SE-minimally positive. 

In the case that $a^1=0$, the constant function $f(z)=a^0$ is the solution of $\partial CF\Pick(\R)$. Therefore, $a^2 =a^3= \dots = a^{2m}=0$. Thus $H_m(a)$ is  SE-minimally positive and $a^{n}$ satisfies (\ref{form_a_2m}).

If $H_m(a)$ is  SE-minimally positive and $a^1>0$ then by Proposition \ref{H_SEmin_posit_a_2m}, $a^{n}$ satisfies (\ref{form_a_2m}). 

{\em Sufficiency} of (1). In the case $H_m(a)>0$,
by Proposition \ref{H_m(a)>0_HD},  the problem $\partial CF\Pick(\R)$ is solvable. 

Consider the case that $a^1=0$. Then 
$H_m(a) \ge 0$ if and only if 
$a^2 = a^3 = \dots = a^{n-1} = 0$ and $a^{n} \ge 0$. $H_m(a)$ is SE-minimally positive, and so $a^{n}=0$.
 Hence the constant function $f(z)=a^0$ is the solution of $\partial CF\Pick(\R)$.

Suppose that $H_m(a)$ is SE-minimally positive, $a^1>0$ and $a^{n}$ satisfies (\ref{form_a_2m}). Then 
by Proposition \ref{H_SEmin_posit_a_2m}, there is an $f\in\Pick$ such that $f^{(k)}(x)/k! = a^k$ for $k=0,1,\dots,n$, that is, $f$ is a solution of Problem $\partial CF\Pick(\R)$.  We have proved (1). Moreover, $f$ is real rational of degree equal to $\rank H_m(a)$.

 (2) Suppose Problem $\partial CF\Pick(\R)$ solvable. By (1) either the associated Hankel matrix $H_{m}(a)$, $n=2m$, is positive definite or both $H_{m}(a)$ is SE-minimally positive of rank $r \ge 1$ and $a^{n}$ satisfies (\ref{form_a_2m}).  In the former case, by  Theorem \ref{H_m(a)>0_HD}, there is a real  rational solution of degree  $\rank H_m(a)$.  In the latter case, if $a^1=0$ then $H_m(a)=0$ and the conclusion follows easily,
whereas, when $a^1>0$,
if $H_m(a)$ is SE-minimally positive and $a^{n}$ satisfies (\ref{form_a_2m}), then by Proposition  \ref{H_SEmin_posit_a_2m},
 the solution $f$ of $\partial CF\Pick(\R)$ is a real  rational function of degree equal to $\rank H_m(a)$. Thus (2) holds for all $m$.

(3)  By Theorem \ref{H_m(a)>0_HD}, if $H_m(a) >0$ then Problem $\partial CF\Pick(\R)$ is indeterminate. By Proposition \ref{H_SEmin_posit_a_2m}, if $H_m(a)$ is SE-minimally positive and  $a^{n}$ satisfies (\ref{form_a_2m}) then the problem $\partial CF\Pick(\R)$ is determinate. 
\end{proof}

We now come to the proof of the main solvability result of the paper.  We make a couple of preliminary observations.

It is well known (for instance, 
\cite[Proposition 3.1]{AY10}) that, for each $f$ in the Pick class, if $f$ is meromorphic and has a pole at $x \in \R$ then $f$ has a {\em simple} pole at $x$, with negative residue. Therefore Problem $\partial CF\Pick$  is only solvable if $a^{-1} \in \R$ and $a^{-1} \le 0$.

Problem $\partial CF\Pick$  is trivial if $n=0$.  It is easy to see from consideration of Taylor expansions that a non-constant function in the Pick class has non-vanishing derivative at any point of $\R$ at which it takes a real value (for example, \cite[Proposition 3.1]{AY10}).   Accordingly we suppose in Problem $\partial CF\Pick$  that $n\geq 1$ and $\im a^0\geq 0$.

{\bf Proof of Theorem \ref{main_theorem}}.~
By Proposition \ref{a-1toa0},
we may reduce to the case that $a^{-1}=0$.

(1) If a function $f \in \Pick $ is analytic at $x$ then the restriction of $\im f$ to a suitable neighbourhood of $x$ in $\R$ is a smooth non-negative real function.  Suppose that $f\in \Pick$ is a solution of Problem $\partial CF\Pick$ and $\rho(a)$ is odd; then $a^0$ is real and so $\im f$ attains its minimum of $0$ at $x$.  The first non-zero derivative of the real function $\im f$ at $x$ is therefore an even derivative, and so the first non-real term of the sequence $a$ has even subscript.  Hence, if $\rho(a)$ is odd then Problem $\partial CF\Pick$  has no solution.

(2)  We consider the case that $\rho(a)$ is finite and even.  The proof will be by induction on $\rho(a)$.  If $\rho(a) = 0$, which is to say that $\im a^0 \neq 0$, then statement (2) holds by virtue of Theorem \ref{a0inPi}.
 Now let $m \geq 1$ and suppose that statement (2) of the theorem holds whenever $\rho(a)\leq 2m-2$.  Consider any sequence $a=(a^0, \dots,a^{2m}, \dots, a^n)$ such that $\rho(a)=2m$.  Suppose that $H_m(a) > 0$ and $\im a^{2m} >0.$  We have $a^1 > 0$.
Let $\sum_0^\infty g_j z^j$ be the reduction of $\sum_0^{n} a^j z^j$ at $0$. By Proposition  \ref{taylorcoeffs},  $g_0,g_1,\dots,g_{n-2} \in \C$ can be written in matrix terms thus:
\beq \label{defgbis}
  \mathrm {a^1} \left[\begin{array}{cccccc} \mathrm {a^1} & 0 & \cdot&\cdot &\cdot & 0 \\ \mathrm {a^2} & \mathrm {a^1} & \cdot&\cdot &\cdot &0\\
  \cdot&\cdot& \cdot& \cdot& \cdot& \cdot\\ \mathrm {a^{2m-1}} & \mathrm {a^{2m-2}} & \cdot & \mathrm {a^1} & \cdot &0 \\
  \cdot & \cdot &\cdot & \cdot &  \cdot & \cdot\\
 a^{n-1} & a^{n-2} & \cdot&\cdot & \cdot& a^1   \end{array}\right]
\left [ \begin{array}{c} g_0\\g_1\\ \cdot\\ g_{2m-3}\\ \cdot \\ g_{n-2}\end{array}\right] =
\left[ \begin{array}{c} \mathrm {a^2} \\ \mathrm {a^3} \\ \cdot \\ \mathrm {a^{2m-1}} \\ \cdot \\ a^n \end{array}\right].
\eeq
Here the entries printed in mathematical Roman font are known to be real.  It follows on truncation to the first $2m-2$ rows that $g_0, g_1, \dots, g_{2m-3}$ are real, and hence $\rho(g_0, g_1, \dots, g_{n-2}) \geq 2m-2$.  From the $(2m-1)$st row of equation (\ref{defgbis}) we have
\[
\mathrm {a^1(a^{2m-1}g_0 + \dots + a^2 g_{2m-3} + a^1}g_{2m-2}) = a^{2m},
\]
and hence, on taking imaginary parts, we find that
\beq \label{im2m}
\im g_{2m-2} = (\im a^{2m})/(a^1)^2  > 0,
\eeq
and so $\rho(g_0, \dots,g_{n-2}) = 2m-2$.
By Corollary \ref{congruent}, the Hankel matrix
$H_{m-1}(g_0,$ $\dots,g_{2m-1},\dots, g_{n-2}) $
is congruent to $\schur H_m(a)$.  Since  $H_m(a)>0$, it is also true that $H_{m-1}(g_0,\dots,g_{2m-1},\dots, g_{n-2}) > 0 $.

By the inductive hypothesis there exists $g \in \Pick$ such that 
\[
\frac{g^{(k)}(x)}{k!} = g_k, \qquad k=0,1,\dots, n-2.
\]
Hence
\[ 
g(z) - \sum_0^\infty g_j (z-x)^j = O((z-x)^{n-1}).
\]
Note that $\sum_0^{n} a^j (z-x)^j$ is the augmentation of $\sum_0^\infty g_j (z-x)^j $ at $x$ by $a^0,a^1$.
Let $f$ be the augmentation of $g$ at $x$ by $a^0,a^1$.
Thus $f\in \Pick$, $f$ is analytic at $x$ and, by Lemma \ref{2augment},
\[ f(z) - \sum_0^{n} a^j (z-x)^j = O(z^{n+1}).
\]
Therefore
$f^{(k)}(x)/k! = a^k$ for $k=0,1,\dots, n$.
Thus $f$ has all the desired properties, and we have proved sufficiency in statement (2) of Theorem  \ref{main_theorem}.

Conversely, suppose that $\rho(a)=2m$ and Problem $\partial CF\Pick$  has a solution $f\in\Pick$.  Let $g(z) = \sum_j g_j(z-x)^j$ be the reduction of $f$ at $x$.  Then by Theorem \ref{propfg} $g\in \Pick$ and $g$ is analytic at $x$.  Again from equation (\ref{defgbis}) we have $\rho(g_0, \dots,g_{n-2})=2m-2.$  By the inductive hypothesis $\im g_{2m-2} > 0$ and $H_{m-1}(g_0,\dots,g_{n-2}) > 0$; hence by Corollary \ref{congruent} (with $n=m-1$), $H_m(a) > 0$, while from equation (\ref{im2m}) we have $\im a^{2m} > 0$.  By induction, statement (2) of Theorem  \ref{main_theorem} holds.

Statements (3) and (4), concerning the case that all $a^j$ are real,
follows from Theorem \ref{sarasonHD_odd}  and Theorem \ref{sarasonHD_even}.

\section{Parametrization of solutions} \label{param}

In this section we give various descriptions of the solution set of  
Problem $\partial CF\Pick$ (when nonempty). The simplest case occurs when the associated Hankel matrix is  SE-minimally positive, for then, by Theorem \ref{sarasonHD_odd}, there is a unique solution.

\begin{theorem} \label{unique_sol_odd}  Let $a=(a^{-1},a^0,\dots,a^{n})\in \R^{n+2}$ where $a^{-1} \le 0$ and suppose that $H_{m}(a)$  is SE-minimally positive and of rank $r \ge 1$, where $m$ is the integer part of $\tfrac 12 (n+1)$.  If $n$ is even, suppose further that $a^{n}$ satisfies condition {\rm (\ref{form_a_2m})}.
The unique solution $F$ of Problem $\partial CF\Pick$ is
\beq \label{form_f}
F(z)=  \frac{a^{-1}}{z-x} + \frac{\sum_{j=0}^r \left(\sum_{k=0}^{j} c_k a^{j-k} \right) (z-x)^{j}}{\sum_{j=0}^r c_j  (z-x)^j}
\eeq 
where $c_0 = -1$ and
\beq \label{form_f_2}
\left[ \begin{array}{cccc}
 c_r  &  c_{r-1} & \dots & c_1\end{array}\right] =
\left[\begin{array}{cccc} a^{r+1} & a^{r+2} & \dots & a^{2r} \end{array}\right] H_r(a)^{-1}.
\eeq 
\end{theorem}

\begin{proof} By Proposition \ref{a-1toa0}, in the case $a^{-1} \le 0$, 
 $$F(z) = \frac{a^{-1}}{z-x} + f(z)$$
where $f$ is the unique solution of Problem $\partial CF\Pick(\R)$; see equation (\ref{interpcond_anal}).

For ease of notation take $x=0$.  Consider first the case of odd $n$. Let the unique solution of Problem 
$\partial CF\Pick(\R)$ be $f(z)= \sum_{k=0}^\infty a^kz^k$.  By Theorem \ref{sarasonHD_odd} $f$ is real rational of degree $r$.  By Kronecker's Theorem \cite[Theorem 3.1]{peller}, the infinite Hankel matrix
\[
H=\left[ a^{i+j-1}\right]_{i,j=1}^\infty
\]
has rank $r$, and hence any $(r+1)$-square
 submatrix of $H$ is singular.  Thus, for $k \geq 1$, 
\begin{align*}
0 &= \det \left[ \begin{array}{cccccccc} 
a^1   & a^2 & \dots  & a^{r}  & a^{k} \\ 
\cdot & \cdot & \dots  & \cdot & \cdot \\
 a^r & a^{r+1} & \dots & a^{2r-1} &  a^{k+r-1}\\ 
 a^{r+1} & a^{r+2} &\dots & a^{2r}  & a^{k+r}
\end{array} \right] \\
  &= \det H_r(a)  \left( a^{k+r} -
\left[ \begin{array}{cccc}  a^{r+1} & a^{r+2} &\dots & a^{2r}\end{array} \right] H_r(a)^{-1} \left[\begin{array}{c} a^{k} \\ a^{k+1}\\ \cdot \\ a^{k+r-1} \end{array}\right] \right) \\
  &=(\det H_r(a))(a^{k+r} - c^r a^k - c_{r-1} a^{k+1} - \dots -c_1 a^{k+r-1}).
\end{align*}
By Lemma \ref{H_SEmin_posit_rank}, $\det H_r(a) \neq 0$, and thus
\beq \label{form_a_k+r}
a^{k+r} = c^r a^k + c_{r-1} a^{k+1} + \dots +c_1 a^{k+r-1}.
\eeq 
On multiplying this equation by $z^{k+r} $ 
and summing over $k=1,2, \dots,$ we have
\[
\sum_{k=r+1}^\infty a^k z^k = c_1z \sum_{k=r}^\infty  a^{k} z^{k} + c_2 z^2\sum_{k=r-1}^\infty  a^{k} z^{k} + \dots + c_r z^r\sum_{k=1}^\infty  a^{k} z^{k}.
\]
That is,
\[
f(z)-a^0-\dots - a^r z^r= c_1z (f(z) - a^0 -\dots - a^{r-1}z^{r-1}) + \dots +c_rz^r (f(z)-a^0).
\]
Since $c_0=-1$ this equation may be written
\[
-\left( \sum_{j=0}^r c_j z^j\right) f(z) = -\sum_{j=0}^r \left(\sum_{k=0}^j c_k a^{j-k}\right) z^j,
\]
and equation (\ref{form_f}) follows.

The case of even $n$ follows from Theorem \ref{unique_sol_odd} and Theorem \ref{sarasonHD_even}.
 \end{proof}

\begin{theorem} \label{parametrizeHD1p}
Let $a=(a^{-1},a^0,\dots,a^n)\in \R^{n+2}$ where $n\geq 1,\, a^{-1} \le 0$. Suppose  $H_m(a) > 0$ where $m$ is the integer part of $\tfrac 12 (n+1)$. 
There exist numbers
\[
s_1,\dots,s_m \in \R, \;\; \text{and, if} \; n \; \text{is even,} \; s_{m+1} \in \R, \;\; t_1,\dots,t_m > 0
\]
such that the general solution of Problem $\partial CF\Pick$  is $f(z)=\frac{a^{-1}}{z-x} + f_1(z)$ where the function $f_1$ is given by the following construction.
\begin{itemize}
\item[\rm(1)] Let $h$ be any function in $\Pick$ that is analytic at $x$;
\item[\rm (2)] if $n$ is odd let $f_{m+1} =h$,
while if $n$ is even let $f_{m+1}$ be the augmentation of $h$ at $x$ by $ s_{m+1}, t$ for some $t >0$;
\item[\rm (3)] $f_k$ is the augmentation of $f_{k+1}$ at $x$ by $s_k, t_k$ for $k=m,\dots,1$, that is,
$$f_k(z)= s_k + \frac{1}{\frac{1}{t_k(z-x)} -f_{k+1}(z)}.$$
\end{itemize}
The numbers $s_j, \ t_j$ are expressible in terms of quantities $a^i(j),\; 0\le i \le n - 2(j-1),$  which are given inductively, for $j=1,\dots,m$, by the equations 
\beq \label{defst}
a^i(1)= a^i,\qquad  0\le i \le n, \qquad 
\eeq
and 
\beq \label{def-a(i)}
   {a^1(j)} \left[\begin{array}{cccccc}  {a^1(j)} & 0 & \cdot&\cdot &\cdot & 0 \\  {a^2(j)} &  {a^1(j)} & \cdot&\cdot &\cdot &0\\
  \cdot & \cdot &\cdot & \cdot &  \cdot & \cdot\\
 a^{n-2j+1}(j) & a^{n-2j}(j) & \cdot&\cdot & \cdot& a^1(j)   \end{array}\right]
\left [ \begin{array}{c} a^0(j+1)\\a^1(j+1)\\ \cdot\\ \cdot \\ a^{n-2j}(j+1)\end{array}\right] =
\left[ \begin{array}{c}  {a^2(j)} \\  {a^3(j)} \\ \cdot \\ \cdot \\ a^{n -2(j-1)}(j) \end{array}\right].
\eeq
Then $s_j, t_j$  are defined by the
equations 
\beq \label{defsty}
s_j = a^0(j),\;\; \text{and, if} \; n \; \text{is even,} \; s_{m+1} = \frac{a^2(m)}{(a^1(m))^2},\qquad t_j = a^1(j).
\eeq
\end{theorem}
\begin{proof}
The proof is a repetition of the proof by induction on $n$ of Theorem \ref{H_m(a)>0_HD}. 
 \end{proof}


\begin{corollary}\label{continued}
Under the hypotheses of Theorem \ref{parametrizeHD1p} the general solution of Problem $\partial CF\Pick$  is
\[
f(z) = \frac{a^{-1}}{z-x}+ s_1+  \frac{1}{\frac{1}{t_1(z-x)} -s_2- }\frac{1}{\frac{1}{t_2(z-x)}-s_3-}\dots \frac{1}{\frac{1}{t_m(z-x)} - f_{m+1}(z)}
\]
where, if $n$ is odd, $f_{m+1}$ is any function in $\Pick$ that is analytic at $x$,
while if $n$ is even, $f_{m+1}$ is the augmentation of $h$ at $x$ by $ s_{m+1}, t$ for some $t >0$, where $h$ is any function in $\Pick$ that is analytic at $x$.
\end{corollary}


\begin{remark} {\rm In the SE-minimally positive  case $H_m(a)$  is
singular and so its rank $r$ is less than $m$. Then 
$t_1, t_2, \dots, t_r$ are positive and $t_{r+1}=\dots=t_{m}=0$, 
$f_{r+1}$ is the constant function $s_{r+1}$ and  the unique 
 solution $f$ of Problem $\partial CF\Pick$ is
\[
f(z) = \frac{a^{-1}}{z-x}+ s_1+  \frac{1}{\frac{1}{t_1(z-x)} -s_2- }\frac{1}{\frac{1}{t_2(z-x)}-s_3-}\dots \frac{1}{\frac{1}{t_{r}(z-x)} - s_{r+1}}.
\]
}
\end{remark}

Quantities closely related to the $t_j$ were familiar to the masters of old, who found expressions for them in terms of Hankel determinants: see for example \cite[Equation (26)]{Nev1922}, which 
Nevanlinna attributes to Stieltjes, Hamburger and Perron. In the present context the analogous formulae are as follows.

\begin{proposition}\label{NevParam} The  parameters
$t_j = a^1(j)$, $ j=1,\dots,m$, satisfy the equations
\beq \label{form_det_Hk}
\det  H_{k}(a) = t_k  \; t_{k-1}^{3}  \; t_{k-2}^{5}  \; \dots  \; t_1^{2k-1}
\eeq
for $k = 1, 2, \dots,m$, and consequently, for $l = 1,2, \dots,r$,
\beq \label{form_t_k_t_k+1}
 t_l \; t_{l+1} = \frac{D_{l-1} D_{l+1}} {\left( D_l \right)^2}
\eeq
where $r$ is the rank of  $H_{m}(a)$, $D_{-1}= D_0=1$ and 
\[ D_l = \det H_l(a), \; \; l = 1,2, \dots,r.\]
Furthermore, for $l = 1,2, \dots,r$,
\beq \label{form_t_l}
 t_l = \frac{D_l}{D_{l-1}^3} \left(\prod_{k=1}^{l-2} D^{(-1)^{k+l}}_{k} \right)^4.
\eeq
\end{proposition}
\begin{proof} Apply Corollary \ref{schur_compl} with 
\[
f(z)=\sum_{k=0}^\infty a^k(j) (z-x)^k, \qquad 
g(z)=\sum_{k=0}^\infty a^k(j+1) (z-x)^k.
\]
We obtain
\beq \label{hank_det_a(j)}
\det H_{n+1}(a(j)) = (t_j)^{2n+1} \det H_{n}(a(j+1))
\eeq
for $n=1,2, \dots, m-j$. Writing $k=j+1$ and iterating this relation we find
\begin{eqnarray*}
\det H_{n}(a(k))  &=& \frac{1}{t_{k-1}^{2n+1} } \det H_{n+1}(a(k-1))
 = \dots \\
&=& \frac{1}{t_{k-1}^{2n+1} \;  t_{k-2}^{2n+3}  \; \dots   \; t_{1}^{2n+2k-3}} \det H_{n+k-1}(a(1)).\\
\end{eqnarray*}
Put $n=1$ to obtain
\[
t_k = \det H_{1}(a(k)) =  \frac{1}{t_{k-1}^{3}  \; t_{k-2}^{5}  \; \dots  \;  t_{1}^{2k-1}} \det H_{k}(a).
\]
which is equation (\ref{form_det_Hk}). 
It follows from  equation (\ref{form_det_Hk}) that
\begin{eqnarray*}
\frac{D_{k-1} D_{k+1}} {\left(D_k\right)^2} 
&=& \frac{(t_{k-1}\; t_{k-2}^{3} \; \dots  \; t_1^{2k-3}) \;\;(t_{k+1}\; t_{k}^{3} \;t_{k-1}^5 \; t_{k-2}^{7} \dots  \; t_1^{2k+1})}
{(t_k  \; t_{k-1}^{3}  \; t_{k-2}^{5}  \; \dots  \; t_1^{2k-1})^2} \\ 
&= &t_{k+1}\; t_k.\\
\end{eqnarray*}
Therefore,
\beq \label{form_t_D_ind}
t_{k+1} = \frac{D_{k-1} D_{k+1}} {\left(D_k\right)^2} \cdot \frac{1}{t_k}.
\eeq
We prove equation (\ref{form_t_l}) by induction on $l$.
Note that $t_1 =\frac{D_1}{D_{0}^3} $, so that (\ref{form_t_l}) holds when $l=1$. Suppose that
\beq \label{form_t_l_ind}
 t_l = \frac{D_l}{D_{l-1}^3} \left(\prod_{k=1}^{l-2} D^{(-1)^{k+l}}_{k} \right)^4. 
\eeq
Then, in view of equation (\ref{form_t_D_ind}),
\begin{eqnarray*} 
t_{l+1} &=& \frac{D_{l-1} D_{l+1}}{\left(D_l\right)^2} \frac{1}{t_l}=\frac{D_{l+1}}{D_{l}^3} \left(\prod_{k=1}^{l-1} D^{(-1)^{k+l+1}}_{k} \right)^4. \\
\end{eqnarray*}
 \end{proof}

We can give an alternative expression for the parametrization of solutions in Theorem \ref{parametrizeHD1p} in terms of a linear fractional transformation.  For any $2\times 2$ matrix $A=[a_{ij}]$ let us denote the corresponding linear fractional transformation by $L[A]$:
\[
L[A](w) = \frac{a_{11} w+a_{12}}{a_{21}w + a_{22}}, \qquad w\in\C\cup \{\infty\}.
\]
In this notation the relationship between $f_{k+1}$ and its augmentation $f_k$ at $x$ by $s_k, t_k$ can be written
\[
f_k(z) = L[A_k(z)] ( f_{k+1}(z))
\]
where, for $k=1,\dots, m,$
\beq \label{defA}
A_k(z) =  \left[\begin{array}{cc} s_k t_k(z-x)\; & \;\;-t_k(z-x)-s_k \\  t_k(z-x) & -1 \end{array}\right].
\eeq
Note that $\det A_k(z) = t^2_k(z-x)^2$.
The recursion for $f$ in Theorem \ref{parametrizeHD1p} becomes:
\begin{eqnarray*}
f(z)&=&\frac{a^{-1}}{z-x}+ f_1(z) = \frac{a^{-1}}{z-x}+L[A_1(z)](f_2(z)) = \frac{a^{-1}}{z-x}+L[A_1(z)A_2(z)](f_3(z)) \\
   &=& \dots = \frac{a^{-1}}{z-x}+L[A_1(z)\dots A_m(z)](f_{m+1}(z)).
\end{eqnarray*}
We therefore arrive at the following linear fractional parametrization.
\begin{corollary} \label{linfrac}
Let $a=(a^{-1},a^0,\dots,a^n)\in \R^{n+2}$ where $n\geq 1,\, a^{-1} \le 0$ and $a^1 \neq 0$. Suppose  $H_m(a) > 0$ where $m$ is the integer part of $\tfrac 12 (n+1)$. Then the general solution of Problem $\partial CF\Pick$  is
\[
f(z)= \frac{a^{-1}}{z-x}+ \ds \frac  {a(z) h(z) + b(z)}{c(z) h(z) + d(z)}
\]
where $a, b, c, d$ are real polynomials of degree at most $m$ satisfying, for some $K>0$,
\[
(ad-bc)(z) =  K (z-x)^{2m}
\]
and given by
\[
\left[\begin{array}{cc} a(z) & b(z) \\ c(z) & d(z) \end{array}\right] =
A_1(z) A_2(z) \dots A_{m}(z),
\]
where $A_k(z)$ is given by equations {\rm (\ref{defA})}, the quantities $s_k,t_k$  are as in Theorem {\rm \ref{parametrizeHD1p}} and 
if $n$ is even then $h$ is any function in $\Pick$ that is analytic at $x$ and  satisfies $h(x) = s_{m+1}$, and
if $n$ is odd then $h$ is any function in $\Pick$ that is analytic at $x$.
\end{corollary}

\section{Nevanlinna's analysis} \label{nevanlinna}

Although Nevanlinna's paper \cite{Nev1922} appears to be the first on the boundary Carath\'{e}odory-Fej\'{e}r problem, it is rarely cited in subsequent work on the topic. Nevertheless, as we show in this paper, the methods he introduced are of sufficient power to prove detailed results on solvability and parametrization for Problem $\partial CF\Pick(\R)$. We believe they have considerable merit in their simplicity and accessibility.

Nevanlinna's own formulation of the problem differs slightly from ours, in that he took the interpolation node to be $\infty$.
Given real  numbers 
$c_0, \dots, c_{2m-1}$ he sought a function $f$ such that $-f$ is   in the Pick class and 
$$ f(z) = c_0 +\frac{c_1 }{z} +  \frac{c_2}{z^2} +\dots + \frac{c_{2m-1} }{z^{2m-1}} + R_{2m}(z)$$
where the remainder $R_{2m}$ satisfies
$$ \lim z^{2m-1} R_{2m}(z) =0$$
as $z \to \infty$ in any sector $ \{z : \varepsilon < \arg z < \pi -\varepsilon \}$ for $\varepsilon \in \left( 0, \frac{\pi}{2} \right)$.
A function $f$ is a solution of Nevanlinna's problem if and only if the function $F(z)= - f( - \frac{1}{z}) $ solves a problem $\partial CF\Pick(\R)$ in the weak sense that the Taylor expansion $\sum_0^{2m-1} (-1)^{j+1} c_j z^j$ differs from $F(z)$ by a term which is $O(z^{2m-1})$ as $z \to 0$ in any nontangential approach region 
$ \{z : \varepsilon < \arg z < \pi -\varepsilon \}$ for $\varepsilon \in \left( 0, \frac{\pi}{2} \right)$.

Nevanlinna's primary interest was the determinacy of solutions of the Stieltjes moment problem, but {\em en route} he discussed solvability criteria for Problem $\partial CF\Pick(\R)$. His Satz I on page 11 describes the recursive procedure for constructing  solutions of the problem (the method we used in the proof of Theorem \ref{probprimeHD} and Part (2) of Theorem \ref{main_theorem}, with appropriate minor changes) and, states, roughly speaking, that the problem has a solution if and only if the recursion works. He does not formally state a solvability criterion in terms of the original data, but in a discussion on the following page he implies that, in our  terminology, Problem $\partial CF\Pick(\R)$ is 
solvable if and only if either the associated Hankel matrix $H_m(c) >0$ or, for some $k \in \{1, \dots, m-1\}$, $\det H_{j}(c)>0$ for $1 \le j \le k$ and $\det H_{k+1}(c)=\dots=\det H_{m}(c)=0$.  He did not give a full proof of this statement, and in fact it is inaccurate. Consider the case $m=3$,
\[
H_3(c) = \left[ \begin{array}{ccc} 1\; & 1\; & 1\; \\ 
   1 \;& 1 \; & 1\; \\ 
   1 \;& 1 \;& 2 \;\end{array} \right].
\]
Here $\det H_{1}(c) >0$ and $\det H_{2}(c) =\det H_{3}(c) =0$, but since $ H_{3}(c)$ is positive, singular and not SE-minimally positive, the corresponding problem  $\partial CF\Pick(\R)$ has no solution.

\section{Conclusion} \label{conclud}

We have presented an elementary and concrete solution of the classical  boundary Carath\'{e}odory-Fej\'{e}r interpolation problem. There are numerous alternative approaches in the literature, many addressing more general interpolation problems, but we believe that our main theorem, Theorem \ref{main_theorem}, is new. Specifically, (a) the notion of SE-minimal positivity of the associated Hankel matrix in the  solvability criterion is new, (b) we allow {\em complex} data 
$ a^0,\dots,  a^{n}$ in  Problem $\partial CF\Pick$ and (c) we allow both even and odd $n$. We comment briefly on some other approaches found in the 
literature.

A very valuable source of information about all manner of complex 
interpolation problems is the book of Ball, Gohberg and Rodman \cite{bgr}.
On pages 473-486 they study the ``generalized angular derivative interpolation (GADI) problem" associated with a data set consisting of a 4-tuple of matrices. They solve  the problem using their highly-developed realization theory of rational matrix functions. Their Corollary 21.4.2 relates to a problem which contains our Problem $\partial CF\Pick(\R)$ as a special case (modulo identification of the disc with a half-plane): one must take $N=1$, $k=0$, $C_{0-}=  \left[ \begin{array}{cccc}
  1\; & 0 \; & \dots &0 \;\end{array}\right]$, $C_{0+}=  \left[ \begin{array}{cccc}
  a^0 & a^{1} & \dots &a^{2m-1}\end{array}\right] $ and $A_0$ to be the $2m$-square Jordan block with eigenvalue $x$. The conclusion, after some work, is that $H_{m}(a) \ge 0$ is necessary and $H_{m}(a)>0$ is sufficient for solvability of $\partial CF\Pick(\R)$. They also give an explicit linear fractional parametrization of the solution set in the case that 
$H_{m}(a)>0$. Since they eschew ``singular" problems throughout the book (here, problems for which $H_{m}(a)$ is singular), they do not obtain a 
 solvability criterion.

The monograph \cite{BD} by Bolotnikov and Dym is entirely devoted to 
boundary  interpolation problems, though for the Schur class rather than the Pick class. They reformulate the problem within the framework of the Ukrainian school's Abstract Interpolation Problem
and solve it by means of operator theory in de Branges-Rovnyak spaces.
They too study a very much more general problem than we do here.
See also \cite[Theorem 4.3]{BolKh08}.

J. A. Ball and J. W. Helton have developed a far-reaching theory of interpolation which they call the ``Grassmanian" approach; it makes use of the geometry of Krein spaces.  It gives a unified treatment of many classical interpolation problems, including the ones studied in this paper -- see \cite[Theorem 6.3]{bh}.  Inevitably, one pays for the generality with a less concrete parametrization.

The treatment that is closest to ours is probably that is of Georgijevi\'c \cite{Geo98}. He allows finitely many interpolation nodes in $\R$ and finitely many derivatives  prescribed at each interpolation node.
He too uses  Nevanlinna-Julia recursion, which he calls the ``Schiffer-Bargmann Lemma"\footnote{This terminology is taken from E. Wigner 
\cite{wig}, who was evidently unaware of the work of Julia and Nevanlinna
30 years earlier. Wigner and J. von Neumann wrote a number of papers in the 1950s on a subclass of the Pick class that arises in the quantum theory of collisions.}, but his methods are less elementary: he uses de 
Branges-Rovnyak spaces and Nevanlinna's integral representation for functions in the Pick class. He  obtains a 
 solvability criterion for Problem $\partial CF\Pick(\R)$, in the case that $n$ is odd, which differs from ours.

We believe that a combination of the ideas in this paper with those in \cite{AY10} will produce a solution to the boundary  interpolation problem
with multiple nodes in $\R$ and with 
 finitely many derivatives  prescribed at each interpolation node, but we have not worked out the details.

JIM AGLER, Department of Mathematics, University of California at San Diego, CA \textup{92103}, USA\\

ZINAIDA A. LYKOVA,
School of Mathematics and Statistics, Newcastle University,
 NE\textup{1} \textup{7}RU, U.K.~~
e-mail\textup{: \texttt{Z.A.Lykova@newcastle.ac.uk}}\\

N. J. YOUNG, School of Mathematics, Leeds University, LS2 9JT, U.K.~~
e-mail\textup{: \texttt{N.J.Young@leeds.ac.uk}}


\begin{thebibliography}{99}

\bibitem{AY10}  {\sc  
 J. Agler and N. J. Young}. Boundary Nevanlinna-Pick interpolation via reduction and augmentation.  \emph{Math. Zeitschrift.} Online DOI: 10.1007/s00209-010-0696-3 (2010).


\bibitem{bgr} {\sc  J. A. Ball, I. Gohberg and L. Rodman}. \emph{Interpolation of Rational Matrix Functions}. Operator Theory: Advances and Applications \textbf{45} (Birkh\"auser Verlag, Basel, 1990).

\bibitem{bh} {\sc  J. A. Ball and J. W. Helton}.  Interpolation problems of Pick-Nevanlinna and Loewner types for meromorphic matrix functions: parametrization of the set of all solutions. \emph{Integral Equ. Oper. Theory} \textbf{ 9} (1986) 155--203.

\bibitem{Bol09} {\sc  V. Bolotnikov}.
On boundary angular derivatives of an analytic self-map of the unit disk. \emph{ C. R. Acad. Sci. Paris}, Ser. I, \textbf{347} (2009) 227--230.

\bibitem{Bol10} {\sc  V. Bolotnikov}.  The boundary analog of the   Carath\'eodory-Schur interpolation problem. arXiv:1008.3364v1 [math.CA] (2010).

\bibitem{BD} {\sc  V. Bolotnikov and H. Dym}.  On boundary interpolation for matrix valued Schur functions. \emph{AMS Memoirs} \textbf{181} (2006) Number 856 1--107.


\bibitem{BolKh08}  {\sc  V. Bolotnikov and A. Kheifets}.
The higher-order Carath\'eodory-Julia theorem and related boundary interpolation problems. \emph{Recent advances in matrix and operator theory}. Operator Theory: Advances and Applications Vol. \textbf{179}  63--102  (Birkh\"auser, Basel, 2008).

\bibitem{CF} {\sc  C. Carath\'eodory and L. Fej\'{e}r}. \"Uber den Zusammenhang der Extremen von harmonischen Funktionen mit ihren Koeffizienten und \"uber den Picard-Landau'schen Satz. \emph{Rend. Circ. Mat. Palermo} \textbf{32} (1911) 218--239. 

\bibitem{car54}
{\sc  C.~Carath\'eodory},
 \emph{Theory of functions, Vol. II}.
(Chelsea Publishing Company, New York, 1954).


\bibitem{Geo98}
 {\sc  D. R. Georgijevi\'c}. Solvability condition for a boundary value interpolation problem of Loewner type. \emph{J. Analyse Math\'ematique} \textbf{74} (1998) 213--234.

\bibitem{Geo05} {\sc  D. R. Georgijevi\'c}. Mixed L$\ddot{\rm o}$wner and Nevanlinna-Pick Interpolation.  \emph{Integral Equations and Operator Theory}  \textbf{53} (2005) 247--267. 

\bibitem{Ju20}  {\sc  G. Julia}.  Extension nouvelle d'un lemme de Schwarz. \emph{Acta Math.} \textbf{ 42}(1920) 349--355.

\bibitem{ko89}
 {\sc  I. V. Kovalishina}. A multiple boundary value interpolation problem for contracting matrix functions in the unit disk. \emph{Teor. Funktsii Funktsional. Anal. Prilozhen. } \textbf{51}  (1989) 38--55 (English trans.: \emph{J. Soviet Math.} \textbf{52} (1990) 3467--3481).

\bibitem{LimAnd} {\sc  D. J. N. Limebeer and B. D. O. Anderson}.
An interpolation theory approach to $H^\infty$ controller degree bounds.
\emph{Linear Algebra and its Applications} \textbf{98} (1988) 347--386.

\bibitem{Mar} {\sc  D. E. Marshall}.  An elementary proof of the Pick-Nevanlinna interpolation theorem.
\emph{Michigan Math. J.} \textbf{21} (1975) 219--223.

\bibitem{Nev1}  {\sc  R. Nevanlinna}. Kriterien f\"ur die Randwerte beschr\"{a}nkter Funktionen. \emph{Math. Zeitschrift} \textbf{13} (1922) 1--9.


\bibitem{Nev1922}   {\sc  R. Nevanlinna}. Asymptotische Entwicklungen Beschr$\ddot{\rm a}$nkter Funktionen und das Stieltjessche Momentenproblem.  \emph{Ann. Acad. Sci. Fenn. Ser. A}  \textbf{18}(5) (1922) 53 pp.

\bibitem{peller} {\sc  V. V. Peller}. \emph{Hankel Operators and their Applications} (Springer Verlag, New York 2003).

\bibitem{power}
 {\sc  S. C. Power}. An elementary approach to the matricial Nevanlinna-Pick interpolation criterion. \emph{Proceedings of the Edinburgh Mathematical
Society. Ser. 2} \textbf{32} (1989) 121--126.


\bibitem{S}   {\sc  D. Sarason}. Nevanlinna-Pick interpolation with boundary data. \emph{Integr. Equ. Oper. Theory} \textbf{30} (1998) 231--250

\bibitem{Schur} {\sc  I. Schur}. \"Uber Potenzreihen die im Innern des Einheitskreises beschr\"ankt sind. \emph{J. Reine Angew. Math.} \textbf{147} (1917) 205-232; \textbf{148} (1918-19) 122-145. 

\bibitem{wig} {\sc  E. Wigner}.  On a class of analytic functions from the quantum theory of collisions. \emph{Ann. Maths} \textbf{53} (1951) 36--59.
\end{thebibliography}
\end{document}